\documentclass[11pt]{amsart}
\usepackage[T1]{fontenc}
\usepackage{lmodern,amsmath,amssymb,amsthm,mathtools}
\usepackage[margin=1in]{geometry}
\usepackage{enumitem,microtype}
\usepackage{xcolor}

\numberwithin{equation}{section}
\usepackage[colorlinks=true,linkcolor=blue,citecolor=blue,urlcolor=blue]{hyperref}

\setlist[enumerate]{label=\textup{(\arabic*)},leftmargin=*,itemsep=4pt}
\newtheorem{theorem}{Theorem}[section]
\newtheorem{lemma}[theorem]{Lemma}
\newtheorem{proposition}[theorem]{Proposition}
\newtheorem{corollary}[theorem]{Corollary}
\theoremstyle{definition}
\newtheorem{definition}[theorem]{Definition}
\newtheorem{remark}[theorem]{Remark}
\newcommand{\sa}{\mathrm{sa}}
\newcommand{\baupre}[1]{\xrightarrow[#1]{\mathrm{b.a.u.}_*}}
\newcommand{\bau}[1]{\xrightarrow[#1]{\mathrm{b.a.u.}}}

\newcommand{\uw}{\mathop{\longrightarrow}\limits^{\mathrm{w^*}}}
\newcommand{\Tr}{\operatorname{Tr}} 
\newcommand{\supp}{\operatorname{s}}

\newcommand{\wstar}{\mathrm{weak}^*}
\newcommand{\mPhi}{\mathfrak m_\Phi}

\newcommand{\one}{\mathbf 1}
\newcommand{\Omb}{(g^{\frac{1}{2}} \varphi g^{\frac{1}{2}})}

\title{Quantum integrals II: Dominated and almost uniform convergence}
\author{David P. Blecher}
\address{Department of Mathematics, University of Houston, Houston, TX
77204-3008, USA}
\email[David P. Blecher]{dpbleche@central.uh.edu}

\date{November 1, 2026}
\begin{document}

\begin{abstract} One main aim of our paper is to 
 formulate noncommutative versions of the Lebesgue dominated convergence theorem in the setting of quantum integration theory on von Neumann algebras. Our starting point is a reformulation of classical dominated convergence in terms of weak* convergence. 
Using our recent sequential-continuity theorem, we obtain
dominated convergence formulations for sequentially normal
states and, under domination by a fixed bounded integrable
majorant, for sequentially normal weights.
 Integrable, possibly unbounded, majorants are treated first in the tracial setting, then for arbitrary von Neumann algebras using Haagerup’s \(L^1\) space, and finally using normal semifinite weights and their GNS representations. For any faithful normal semifinite weight \(\varphi\), weak* convergence of norm-bounded sequences is characterized by `integrating  against' bounded positive integrable majorants. We also study noncommutative almost uniform and local almost uniform convergence, and make the relationship with the order convergence results of Chetcuti et al.\  explicit. In particular, for norm-bounded sequences, local \(\varphi\)-bilaterally almost uniform convergence implies weak* convergence and
\(\|h^{1/2}(x_n-x)h^{1/2}\|_1\to0\)
for every positive Haagerup \(L^1\) density \(h\), without assuming \(\sigma\)-finiteness. We also prove a dominated convergence theorem in the predual: fixed positive order domination and Petz’s almost uniform convergence imply convergence in the predual norm. 
We establish implications and equivalences among these convergence notions, connecting dominated convergence with tracial and nontracial integration. \end{abstract}

\maketitle 

\section{Introduction} 

Let $M$ be a complex von Neumann algebra with canonical predual
$M_*$. A state $\varphi$ on $M$ is \emph{countably additive on
projections} if
\[
 \varphi\left(\sum_{n=1}^{\infty}p_n\right)
   =\sum_{n=1}^{\infty}\varphi(p_n)
\]
for every sequence of mutually orthogonal projections in $M$.
It is \emph{sequentially normal}, or \emph{$\sigma$-normal}, if
$\varphi(a_n)\uparrow\varphi(a)$ whenever $0\leq a_n\uparrow a$ in
$M$. Normality requires the analogous assertion for bounded
increasing nets, while \emph{complete additivity on projections}
requires the displayed identity to hold for arbitrary families
of mutually orthogonal projections. (For background on quantum measures and states see Hamhalter's excellent
text  \cite{Hamhalter}.) 
For positive functionals,
complete additivity on projections, normality, and weak*
continuity are equivalent \cite{Tak1}. These functionals form the
positive cone $(M_*)_+$, whose complex linear span is the entire
predual $M_*$. The following sequential analogue was recently proved
in \cite{BLd}. 

\begin{theorem}\label{main}
For a state $\varphi$ on a von Neumann algebra $M$, the following
are equivalent:
\begin{enumerate}[label=\textup{(\roman*)}]
\item $\varphi$ is countably additive on projections.
\item $\varphi$ is sequentially normal.
\item $\varphi$ is sequentially weak$^*$ continuous: if $a_n\to a$
weak$^*$ in $M$, then $\varphi(a_n)\to\varphi(a)$.
\end{enumerate}
In any faithful normal concrete representation of $M$, these
conditions are also equivalent to sequential continuity for the
weak and strong operator topologies.
\end{theorem}

Sequential weak* convergence plays a large role in many branches of mathematics and theoretical statistics.
For example it is often crucially related to  pointwise convergence or a.e.\ convergence in function theoretic settings,
such as Hardy spaces.
Sequential weak* convergence often admits a concrete function-theoretic description; for example, on norm-bounded sequences in \(H^\infty(\mathbb D)\), it is equivalent to pointwise convergence on the disk. 
As we observed in the introduction to \cite{BLd}, the classical
Lebesgue dominated convergence theorem (LDCT) may be viewed as a
statement about sequential weak* convergence. Indeed suppose that $\mu$
is a localizable measure (this restriction is removed in
Section~\ref{class}).  The classical LDCT is equivalent to the assertion
that every norm-bounded sequence in $M=L^\infty(X,\mu)$ that
converges a.e.\ also converges weak*. Indeed, if $(f_n)$ is such a
sequence with a.e.\ limit $f$, then LDCT gives
\[
 \int_X f_n g\,d\mu \longrightarrow \int_X fg\,d\mu
 \qquad \bigl(g\in L^1(X,\mu)\bigr),
\]
which is precisely weak* convergence. Conversely, suppose that (not necessarily essentially bounded) 
$f_n\to f$ a.e.\ and $|f_n|\leq g$ a.e.\ for some nonnegative
$g\in L^1(X,\mu)$. Define $h_n = 0$ when $g = 0$ and otherwise $h_n = f_n/g$; and 
define $h$ similarly. Then $(h_n)$ lies in the
unit ball of $M$ and converges a.e.\ to $h$. If $h_n\to h$ weak*,
then
\[
 \int_X f_n\,d\mu
   =\int_X h_n g\,d\mu
   \longrightarrow
   \int_X hg\,d\mu
   =\int_X f\,d\mu,
\]
since integration against $g\,d\mu$ is a weak* continuous
functional. This gives the conclusion of LDCT.   Indeed applying the same bounded convergence assertion to $|h_n-h|$
also gives
\[
 \|f_n-f\|_1
   =\int_X g|h_n-h|\,d\mu\longrightarrow0.
\]

The sequential-continuity theorem of \cite{BLd}, recalled above,
provides our starting point. We apply it for example  to obtain dominated
convergence in the noncommutative (or nc for short) case for sequentially normal weights:
Proposition~\ref{bounded-weight} shows that, if $\Phi$ is such
a weight, $g\in M_+$ satisfies $\Phi(g)<\infty$, and
$-g\le f_n\le g$ with $f_n\to f$ weak*, then
$\Phi(f_n)\to\Phi(f)$. Here the integrable majorant is bounded,
while the weight need not be normal. We subsequently treat
possibly unbounded majorants in the tracial and Haagerup
$L^1$ settings, and through the GNS representation of a
normal semifinite weight. (Instead of Haagerup $L^1$-spaces one could often use the standard form, if preferred.  Sometimes however this would make statements of results more cumbersome.) 

We often state results in terms of a fixed faithful normal semifinite weight 
$\varphi$. 
We recall (see e.g.\ \cite[Corollary 14.2]{BLd}) that a faithful $\sigma$-normal weight is actually normal.
Thus most of the many hypotheses in our paper using faithful normal weights could equivalently be stated using
$\sigma$-normality with no changes in the proofs. 
For such a weight, 
Theorem~\ref{ldct-merged:criterion} for example,  identifies a particularly
simple family of tests for weak* convergence. For norm-bounded
sequences,
\[
 x_n\longrightarrow x\ \text{weak*}
 \quad\Longleftrightarrow\quad
 \varphi(g^{1/2}x_ng^{1/2})
 \longrightarrow
 \varphi(g^{1/2}xg^{1/2})
 \quad
 \bigl(g\in M_+,\ \varphi(g)<\infty\bigr).
\]
Thus bounded positive integrable majorants suffice. We also
give formulations using positive affiliated majorants.

The final sections concern noncommutative almost uniform
and local almost uniform convergence. Classically, local
almost uniform convergence means almost uniform convergence
on every measurable set of finite measure. On any measure
space, pointwise almost-everywhere convergence implies local
almost uniform convergence by Egoroff's theorem; the converse
holds when the measure is semifinite. Thus, for localizable
$\mu$, the classical LDCT is equivalent to the assertion that
every norm-bounded sequence in $L^\infty(X,\mu)$ which
converges locally almost uniformly also converges weak*
to the same limit. This provides the motivation for the
noncommutative local convergence results below.
 Our local almost uniform convergence results give concrete
hypotheses for convergence of integrals.
Theorem~\ref{ldct-addon:theorem} shows that, for norm-bounded
sequences, local $\varphi$-b.a.u.\ convergence implies weak*
convergence and, more strongly,
\[
 \|h^{1/2}(x_n-x)h^{1/2}\|_1\longrightarrow0
 \qquad\bigl(h\in L^1(M)_+\bigr),
\]
where $L^1(M)$ is Haagerup's space. No $\sigma$-finiteness
assumption on $M$ is required. Local $\varphi$-a.u.\
convergence gives the stronger operator-topology conclusion
of $\sigma$-strong convergence.  Many of the just mentioned results  are generalizations
of the classical dominated convergence theorem.

The one-sided local convergence conclusion is also a consequence of
the order-convergence results recorded in
\cite[\S2.2, p.~103, and Proposition 2.7]{CHW}: order convergence
implies $\sigma$-strong, hence weak*, convergence, without
$\sigma$-finiteness. The identification of our bounded local a.u.\
condition with order convergence of $|x_n-x|$ is proved in
Appendix~\ref{app:order-local}. The precise local-weight formulation which we give 
is not stated in \cite{CHW}.

We also establish a dominated convergence theorem directly in the
predual $M_*$ of an arbitrary von Neumann algebra
(Theorem~\ref{predual:dct}). For selfadjoint normal functionals dominated
in order by a fixed positive normal functional, the almost uniform
convergence of Definition~\ref{predual:definition} implies convergence
in the predual norm, and hence convergence of their values at the
identity.  In the commutative case, the convergence hypothesis
is equivalent to almost everywhere convergence, and the theorem
recovers classical $L^1$ dominated convergence.

Finally, Proposition~\ref{lm-bounded-comparison} characterizes
$\sigma$-strong convergence of norm-bounded sequences by
one-sided local convergence in measure.
The corresponding symmetric condition characterizes
$\sigma$-strong* convergence. In particular, these bounded
sequential convergence conditions are independent of the
chosen faithful normal semifinite weight.

For the readers convenience, we end this introduction with a summary of the implications between various notions of sequential convergence 
of functions on a measure space $(X,\mu)$. Later in our paper we will give analogues of the above in the noncommutative setting (they are summarized
at the end of our paper).  Many other authors have established similar results for special classes of von Neumann algebras
(see e.g.\ \cite{BLS} and references therein, such as the papers of Paszkiewicz cited there).   Namely, almost uniform
convergence implies a.e.\ convergence, indeed the latter is equivalent to local almost uniform
convergence if the measure is semifinite.  
Thus, for a
semifinite measure $\mu$, the whole-sequence implications are
\[
 \text{a.u.}
 \ \Longrightarrow\
 \text{a.e.}
 \ \Longleftrightarrow\
 \text{local a.u.}
 \ \Longrightarrow\
 \text{local convergence in measure}.
\]
The first implication is an equivalence when $\mu$ is finite. 
If the sequence and its limit are dominated in absolute value
by a common function in $L^1(\mu)_+$, then a.u., a.e., and local
a.u.\ convergence are equivalent without any semifiniteness
assumption on the measure.
If $\mu$ is $\sigma$-finite, local convergence in measure implies existence of a
subsequence converging a.e.,
equivalently locally almost uniformly, to the same limit.
For norm-bounded sequences in $L^\infty(X,\mu)$ with $\mu$
localizable, local convergence in measure is equivalent to strong,
strong*, $\sigma$-strong, and $\sigma$-strong* convergence in any
faithful normal representation. These conditions imply weak
operator convergence, which is equivalent to weak* convergence
on norm-bounded sequences.

\section{The classical formulations} \label{class}

Let $(X,\Sigma,\mu)$ be a measure space, not necessarily localizable. Boundedness of a sequence in
$L^\infty(\mu)$ of course means  a uniform essential bound. For an arbitrary
measure space we write $\sigma(L^\infty,L^1)$ explicitly; we do not assume
without qualification that $(L^1)^*=L^\infty$, as is true in the localizable case.

\begin{definition}
A positive functional $\psi$ on a monotone sequentially complete $C^*$-algebra
is \emph{sequentially normal} if
\[
 0\le a_n\uparrow a\quad\Longrightarrow\quad\psi(a_n)\uparrow\psi(a).
\]
The sequences here are bounded and $a$ is their supremum in the algebra.
\end{definition}

The following is no doubt well known \cite[Lemma 3.2]{BGL}.

\begin{lemma}[Measures representing sequentially normal states]\label{measure} 
Sequentially normal states on $L^\infty(\mu)$ correspond bijectively to
countably additive probability measures $\nu$ on $\Sigma$ such that
$\mu(E)=0$ implies $\nu(E)=0$. The correspondence is
\[
 \nu(E)=\psi(\mathbf1_E),\qquad \psi(h)=\int_Xh\,d\nu.
\] 
\end{lemma}

\begin{lemma}[Classical dominated convergence formulations]\label{classical}
The following are equivalent convergence principles:
\begin{enumerate}
\item If $f_n\to f$ almost everywhere and $|f_n|\le g$ almost everywhere
for some $g\in L^1(\mu)_+$, then $f\in L^1(\mu)$ and
$\|f_n-f\|_1\to0$.
\item If $(h_n)$ is bounded in $L^\infty(\mu)$ and $h_n\to h$ almost
everywhere, then $h\in L^\infty(\mu)$ and
$h_n\to h$ in $\sigma(L^\infty,L^1)$.
\item Under the hypotheses in {\rm (2)}, $\psi(h_n)\to\psi(h)$ for every
sequentially normal state $\psi$ on $L^\infty(\mu)$.
\end{enumerate}
All three assertions hold on every measure space.
\end{lemma}
\begin{proof}
For $(1)\Rightarrow(2)$, let $C=\sup_n\|h_n\|_\infty$. Outside a common
null set, $|h_n|,|h|\le C$. For $a\in L^1(\mu)$,
\[
 a(h_n-h)\to0\quad\text{a.e.},\qquad |a(h_n-h)|\le2C|a|.
\]
Assertion (1) gives convergence of the integrals.

For $(2)\Rightarrow(1)$, $|f|\le g$ a.e., and the functions
\[
 u_n=\begin{cases}|f_n-f|/g,&g>0,\\0,&g=0,\end{cases}
\]
are bounded by $2$ and converge to zero a.e..   Integrating  (2) against $g$ gives
$\|f_n-f\|_1=\int gu_n\,d\mu\to0$.

For (3), Lemma~\ref{measure} and dominated convergence for the finite
measure $\nu$ give the assertion immediately. There is also a direct proof
on the fixed algebra, avoiding a change of measure: let
\[
 v_n=\sup_{k\ge n}|h_k-h|.
\]
Then $v_n\downarrow0$ in $L^\infty(\mu)$ and
$|\psi(h_n-h)|\le\psi(v_n)\to0$ by sequential normality.
Finally, $(3)\Rightarrow(2)$ follows by using the sequentially normal
states
\[
 \psi_a(b)=\frac{\int ab\,d\mu}{\int a\,d\mu},
 \qquad 0\ne a\in L^1(\mu)_+,
\]
and decomposing an arbitrary $L^1$ function into a linear combination of
positive functions.
\end{proof}

Here monotone convergence suffices to show that each $\psi_a$ is
sequentially normal. Thus the direct  proof of (3),
followed by $(3)\Rightarrow(2)\Rightarrow(1)$,  gives a proof of
classical LDCT using only monotone convergence and elementary
properties of the integral.

\begin{remark}
The representing measure in Lemma~\ref{measure} need not have an
$L^1(\mu)$ density on an arbitrary measure space. This is why (3) is not
merely a restatement of (2). Nevertheless, (2) and
(3) give the same conclusion for bounded almost-everywhere convergent sequences. 
\end{remark}

The restatements of LDCT above (and in the introduction) may also be rephrased in a way that gives a different 
perspective, and suggests another route forward.  
Suppose again that we are given $g \in L^1(X,\mu)_+$ and  $f_n \to f$ satisfying the usual hypotheses  of LDCT.   Let $d \nu = g d \mu$, so that $\nu \ll \mu$.  We may then do a change of measure, and move the  
problem inside the von Neumann algebra $N = L^\infty(X,\nu)$. 
  The canonical $*$-homomorphism $\pi : L^\infty(X,\mu) \to L^\infty(X,\nu)$  is easily seen to be surjective. 
  Let 
  $(k_n)$ denote the image under $\pi$ of the sequence $(h_n)$ considered in the introduction, similarly $k = \pi(h)$.  
  The dominated-convergence hypotheses on $(f_n)$  are equivalent to the bounded-convergence 
  $k_n \to k$ $\nu$-a.e..   Thus the general LDCT is equivalent to its special case  of a finite measure and an essentially bounded sequence of functions.  
  
The noncommutative variant of this finite measure case 
consists of a faithful normal state on a $\sigma$-finite von Neumann algebra $N$, together with a bounded sequence $(x_n)$ in $N$.  The n.c.\ variant of the convergence of $x_n$ to $x$ which replaces 
a.e.\ convergence, is the almost uniform convergence which we will study in the last sections of our paper
(see also \cite{BLS} and references therein).  Indeed nc almost uniform convergence  in $\sigma$-finite algebras
has been well studied,  e.g.\ in the papers of Paszkiewicz cited there such as \cite{Pasz}.   The associated nc LDCT in this perspective  is then essentially 
sequential nc a.u.\ convergence implying weak* convergence in $N$, which follows from 
\cite[Theorem 1.1.3]{Jaj}. 

\section{Preliminary convergence criteria}
We record consequences of Theorem~\ref{main} and some
bounded convergence criteria used below. Throughout,
weak* convergence in $M$ means convergence in
$\sigma(M,M_*)$.

\begin{lemma}\label{one-state}
For a state $\psi$ on $M$, the following are equivalent:
\begin{enumerate}
\item $\psi$ is sequentially normal.
\item If $(x_n)$ is norm bounded and $x_n\to x$ $\sigma$-strong, then
$\psi(|x_n-x|)\to0$.
\item If $x_n\uw x$, then $\psi(x_n)\to\psi(x)$.
\item Whenever a norm-bounded sequence \(a_n\in M_+\) converges \(\sigma\)-strong to \(a\), 
\[
\psi(a)\le\liminf_n\psi(a_n).
\tag{Fatou}
\]
\end{enumerate}
Moreover, for a bounded sequence $(x_n)$ in $M$ and $x\in M$,
\[
 x_n\uw x\quad\Longleftrightarrow\quad
 \psi(x_n)\to\psi(x)\ \text{for every sequentially normal state }\psi.
\]
\label{all-states}
\end{lemma}
\begin{proof} 
Theorem~\ref{main} gives $(1)\Rightarrow(3)$.

$(3)\Rightarrow(2)$. Put $y_n=x_n-x$.
Since $y_n\to0$ $\sigma$-strongly, for every
$\rho\in M_*^+$ we have
\[
 \rho(|y_n|^2)=\rho(y_n^*y_n)\longrightarrow0.
\]
Thus $|y_n|\to0$ $\sigma$-strongly, hence weak*.
Applying (3) gives $\psi(|y_n|)\to0$.

$(2)\Rightarrow(1)$. If $0\le a_n\uparrow a$,
then $a_n\to a$ $\sigma$-strong,  and
$\psi(a)-\psi(a_n)=\psi(|a-a_n|)\to0$.

The implication $(2)\Rightarrow(4)$ is immediate, since $-|a_n -a| \leq a_n -a  \leq |a_n -a|$.
For $(4)\Rightarrow(2)$, choose
$C\ge\sup_n\|x_n-x\|$. We observed that $|y_n| \to 0$, and so 
\[
 C1-|x_n-x|\longrightarrow C1
 \qquad \sigma\text{-strongly}.
\]
These elements are positive and norm bounded, so (4) gives
\[
 C=\psi(C1)
 \le \liminf_n\psi(C1-|x_n-x|)
 =C-\limsup_n\psi(|x_n-x|).
\]
Since $\psi(|x_n-x|)\ge0$, it follows that
$\psi(|x_n-x|)\to0$, proving (2).  The final equivalence follows from $(1)\Rightarrow(3)$
and the fact that normal states are sequentially normal
and their linear span is $M_*$. \end{proof}

\begin{remark}[Compatibility with the classical formulations]
Suppose that $\mu$ is localizable and that $(h_n)$ is
norm bounded in $L^\infty(\mu)$, with $h_n\to h$ a.e.
Applying Lemma~\ref{classical}(2) to $|h_n-h|^2$ gives
\[
\int a|h_n-h|^2\,d\mu\longrightarrow0
\qquad(a\in L^1(\mu)_+).
\]
Thus $h_n\to h$ $\sigma$-strongly, and
Lemma~\ref{one-state} gives convergence under every
sequentially normal state. The independent recovery of
classical dominated convergence is given in
the direct proof in Lemma~\ref{classical}, and also in
Section~\ref{au-weight-section}.
\end{remark}

The following probably well known criteria are useful descriptions of stronger convergence
hypotheses.
 In particular,
the spectral-projection condition below is exactly $\sigma$-strong
convergence on norm-bounded sequences, and is a nice variant of `convergence in measure'.

\begin{lemma}[Equivalent bounded convergence criteria]
\label{prop:projection-measure-convergence}
Let $(x_n)\subseteq M$ be norm bounded and $x\in M$. The following
conditions are equivalent:
\begin{enumerate}[label=(\roman*)]
\item For every $\delta>0$, the spectral projections
 $e_n(\delta)=\one_{(\delta,\infty)}(|x_n-x|)$ 
 converge weak$^*$ to zero.
\item $x_n\to x$ $\sigma$-strongly.
\item $x_n\to x$ and $x_n^*x_n\to x^*x$ weak$^*$.
\end{enumerate}
Under these conditions, every sequentially normal state $\varphi$
satisfies $\varphi(x_n)\to\varphi(x)$.
\end{lemma}
\begin{proof}
Put $y_n=x_n-x$ and $C=\sup_n\|y_n\|$. Functional calculus gives
\[
 \delta^2 \, e_n(\delta)\leq y_n^*y_n
 \leq\delta^2 \, 1+C^2 \, e_n(\delta).
\]
By definition, (ii) means that $\omega(y_n^*y_n)\to0$ for every
normal positive functional $\omega$. The left inequality proves
(ii)$\Rightarrow$(i). Conversely, (i) and the right inequality give
\[
 \limsup_n\omega(y_n^*y_n)\leq\delta^2\omega(1).
\]
Letting $\delta\downarrow0$ proves (ii). Considering  normal positive
functionals suffices for weak$^*$ convergence, since every normal
functional is a linear combination of such functionals.

Assume (ii). Cauchy--Schwarz gives
\[
 |\omega(y_n)|^2\leq\omega(1) \, \omega(y_n^*y_n),\qquad
 |\omega(x^*y_n)|^2\leq\omega(x^*x)  \,\omega(y_n^*y_n).
\]
Thus $y_n$, $x^*  \, y_n$, and $y_n^*  \, x$ converge weak$^*$ to zero, as
does $y_n^*y_n$. Expanding
\[
 x_n^*x_n-x^*x=y_n^*y_n + x^*  \, y_n+ y_n^*  \, x
\]
proves (iii). Conversely, if (iii) holds  then weak$^*$ continuity of the
adjoint and of multiplication by a fixed element gives
\[
 y_n^*y_n=x_n^*x_n-x_n^*  \, x-x^*  \, x_n+x^*x\longrightarrow0
\]
weak$^*$, which is (ii). The assertion about $\varphi$ follows from Theorem \ref{main}.
\end{proof}

\section{Variants of the dominated convergence theorem}
\label{sec3}

\subsection{Bounded domination for sequentially normal weights} \label{s3was}

A weight $\Phi:M_+\to[0,\infty]$ is additive and positively homogeneous.
It is sequentially normal if it preserves suprema of bounded increasing
positive sequences. Write
$\mathfrak m_\Phi=\operatorname{span}\{a\in M_+:\Phi(a)<\infty\}$;
on this space $\Phi$ has its usual linear extension.

\bigskip

{\bf Remark.} The basic noncommutative Fatou lemma is weak* lower semicontinuity of a normal weight \cite{Haag}: if \(\varphi\) is a normal weight on \(M\), then
\[
x_n,x\in M_+,\qquad x_n\longrightarrow x\text{ weak*}
\quad\Longrightarrow\quad
\varphi(x)\le\liminf_n\varphi(x_n).
\] Applying this to \(g+f_n\) and \(g-f_n\), when \(\varphi(g)<\infty\), gives the dominated convergence conclusion of the following proposition for normal weights. The proposition extends that conclusion to sequentially normal weights.

\begin{proposition}[Bounded domination for sequentially normal weights]  \label{bounded-weight}
Let $\Phi$ be a sequentially normal weight and let $g\in M_+$ satisfy
$\Phi(g)<\infty$. If $f_n=f_n^*\in M$, $-g\le f_n\le g$, and
$f_n\uw f$, then $-g\le f\le g$, all these elements belong to
$\mathfrak m_\Phi$, and $\Phi(f_n)\to\Phi(f)$.
\end{proposition}
\begin{proof}
The order interval is weak* closed. Membership in
$\mathfrak m_\Phi$ follows from $0\le g\pm f_n\le2g$, and similarly
for $f$. Put $p=\supp(g)$. The bounded quadratic-form factorization
of an element of $[-g,g]$, a consequence of Douglas's theorem
\cite[Theorem 1]{Douglas} and its von Neumann algebra version
\cite[Section 2, p.~3 of the arXiv version]{MKX}, gives unique selfadjoint contractions
$h_n,h\in pMp$ such that
\[
 f_n=g^{1/2}h_ng^{1/2},\qquad f=g^{1/2}hg^{1/2}.
\]
The map $h\mapsto g^{1/2}hg^{1/2}$ is weak* continuous and injective
on $pMp$. On its compact selfadjoint unit ball it is a homeomorphism onto
its image. Consequently $h_n\uw h$.

If $\Phi(g)>0$, then
\[
 \rho(a)=\Phi(g)^{-1}\Phi(g^{1/2}ag^{1/2}),\qquad a\in pMp,
\]
is a sequentially normal state. Multiplication preserves the relevant
increasing suprema, so we have verified sequential normality.
Theorem~\ref{main} now proves the result. If $\Phi(g)=0$, all the
displayed integrals vanish by order domination.
\end{proof}

\begin{remark}
Conversely, this proposition, asserted for all sequentially normal weights,
implies the final equivalence in Lemma~\ref{all-states}, by taking
$\Phi=\psi$, $g=C1$ with $C>0$ large enough, and treating real
and imaginary parts separately. 
\end{remark}

\medskip

To allow the majorant to vary, we use {\em sequential weak*
lower semicontinuity} 
\[
a_n,a\in M_+,\quad a_n\to a\ \text{weak*}
\quad\Longrightarrow\quad
\Phi(a)\le\liminf_n\Phi(a_n). \]
Corollary~\ref{cor:weights} gives classes of weights
satisfying this hypothesis.

\begin{theorem}[Varying-majorant dominated convergence]\label{thm:varying}
Suppose that $\Phi$ is sequentially weak* lower semicontinuous on $M_+$.
Let $g_n,g\in M_+$ and $f_n=f_n^*\in M$ satisfy
\[
 g_n\longrightarrow g\ \wstar,\qquad
 f_n\longrightarrow f\ \wstar,\qquad
 -g_n\le f_n\le g_n.
\]
Assume also that for all $n$ we have 
\[
 \Phi(g_n)<\infty ,\qquad
 \Phi(g_n)\longrightarrow\Phi(g)<\infty.
\]
Then $f=f^*$, $-g\le f\le g$, all $f_n$ and $f$ belong to $\mPhi$, and
\[
 \Phi(f_n)\longrightarrow\Phi(f).
\]
\end{theorem}

\begin{proof}
The selfadjoint part and the positive cone of $M$ are weak* closed.  $g_n\pm f_n\ge0$, taking
limits  gives $f=f^*$ and $g\pm f\ge0$.
Since $0\le g_n\pm f_n\le2g_n$, these positive elements have finite
weight, and $f_n\in\mPhi$. Similarly $f\in\mPhi$.

Put $c_n=\Phi(g_n)$, $c=\Phi(g)$, $b_n=\Phi(f_n)$, and $b=\Phi(f)$.
Then $|b_n|\le c_n$, so $(b_n)$ is bounded. Applying sequential weak*
lower semicontinuity to the two positive sequences $g_n+f_n$ and $g_n-f_n$
gives
\[
 c+b\le\liminf_n(c_n+b_n)=c+\liminf_n b_n,
\]
and
\[
 c-b\le\liminf_n(c_n-b_n)=c-\limsup_n b_n.
\]
All terms involving $c$ are finite. Hence
\[
 \limsup_n b_n\le b\le\liminf_n b_n,
\]
which proves the assertion.
\end{proof}

\begin{corollary}\label{cor:weights}
The conclusion of Theorem {\rm \ref{thm:varying}} holds in either of the following cases:
\begin{enumerate}[label=\textup{(\roman*)},leftmargin=*]
\item $\Phi$ is a normal weight;
\item $\Phi$ is sequentially normal and satisfies the condition 
\[
 {\rm (R)}:\qquad
 \Phi(x)=\sup\{\Phi(y):0\le y\le x,\ \Phi(y)<\infty\}
 \quad(x\in M_+).
\]
\end{enumerate}
\end{corollary}

\begin{proof}
In (i), a normal weight is weak* lower semicontinuous by
Haagerup's theorem~\cite[Theorem 1.8]{Haag}.
In (ii), \cite[Theorem 6.22]{BLd} gives sequential weak* lower semicontinuity. 
Then apply Theorem~\ref{thm:varying}.
\end{proof}

\begin{corollary}[Absolute deviations]\label{cor:absolute}
Let $\Phi$ satisfy the hypothesis of Theorem {\rm \ref{thm:varying}}, and similarly let
$g_n,g$ satisfy $g_n,g\in M_+, g_n\to g$ weak*, with $\Phi(g_n)<\infty$ and 
 $\Phi(g_n)\to \Phi(g)<\infty$. Suppose that $(x_n)$ is
norm bounded, $x_n\to x$ $\sigma$-strongly, and
\[
 |x_n-x|\le g_n\qquad(n\ge1).
\]
Then $\Phi(|x_n-x|)\to0$.
\end{corollary}

\begin{proof}
Put $y_n=x_n-x$. For every normal positive functional $\rho$, the
Cauchy--Schwarz inequality gives
\[
 \rho(|y_n|)^2\le\rho(1)\rho(y_n^*y_n)\longrightarrow0.
\]
Since the normal positive functionals linearly span $M_*$, it follows that
$|y_n|\to0$ weak*. Apply Theorem~\ref{thm:varying} with $f_n=|y_n|$ and $f=0$.
\end{proof}

\subsection{An unbounded integrable majorant} \label{oldsec4}

Let $\tau$ be a faithful normal semifinite trace on $M$.
Products of $\tau$-measurable affiliated operators below mean their
closures in the algebra of measurable operators. 

In this subsection, we formulate dominated convergence
through bounded normalized sequences. For a positive integrable
majorant $g$, we write
$f_n=g^{1/2}x_ng^{1/2}$ and $f=g^{1/2}xg^{1/2}$,
where the convergence hypotheses are imposed on $x_n,x\in M$.
For selfadjoint contractions these products satisfy
$-g\le f_n,f\le g$.
Classically, this is the normalization $x_n=f_n/g$ discussed in the introduction,
with the quotient set equal to zero where $g=0$. 

\label{HL1}

For a general von Neumann algebra $M$, we use the Haagerup spaces $L^p(M)$ and their canonical integration functional
$\Tr:L^1(M)\to\mathbb C$. The relevant standard facts are
$L^1(M)\cong M_*$, the pairing $\omega_g(a)=\Tr(ga)$, and H\"older's
inequality for products in these spaces. See \cite[Section 1.2]{HJX}.
Here $\Tr$ is an integration functional on $L^1(M)$, not a claim that
$M$ has a faithful normal semifinite trace. The unbounded operators are
realized in the crossed-product construction, not necessarily as
operators affiliated with $M$ itself.

\begin{theorem}[Unbounded dominated case without a trace on $M$]\label{haagerup}\label{tracial}\label{trace-norm}
Let $M$ be any von Neumann algebra, and let $x_n,x\in M$ be selfadjoint
contractions. The following are equivalent:
\begin{enumerate}
\item $x_n\uw x$.
\item $\psi(x_n)\to\psi(x)$ for every sequentially normal state on $M$.
\item For every $g\in L^1(M)_+$, put
\[
 f_n=g^{1/2}x_ng^{1/2},\qquad f=g^{1/2}xg^{1/2}.
\]
Then $\Tr(f_n)\to\Tr(f)$.
\end{enumerate}
In {\rm (3)}, $f_n,f\in L^1(M)$ and $-g\le f_n,f\le g$.
If, in addition, $x_n\to x$ $\sigma$-strongly, then
$\|f_n-f\|_1\to0$ for each $g$.
More generally, this $L^1$-norm conclusion holds for any norm-bounded
sequence $(x_n)$ in $M$ converging $\sigma$-strongly to $x$, without
assuming selfadjointness or contractivity. It also holds if instead
$x_n^*\to x^*$ $\sigma$-strongly.
The same assertions hold in $L^1(M,\tau)$, with $\tau$ in place
of $\Tr$, when $\tau$ is a faithful normal semifinite trace. \end{theorem}

\begin{proof}
H\"older's inequality gives
$\|g^{1/2}ag^{1/2}\|_1\le\Tr(g)\|a\|$ for $a\in M$,
and positivity gives the stated order bounds.
The functionals
\[
\omega_g(a)=\Tr(g^{1/2}ag^{1/2})=\Tr(ga),
\qquad g\in L^1(M)_+, \]
are precisely the normal positive functionals on $M$.
This proves $(1)\Leftrightarrow(3)$, while
Lemma~\ref{all-states} gives $(1)\Leftrightarrow(2)$.
For the norm conclusion, put $y_n=x_n-x$ and
$\omega_g(a)= \Tr(g^{1/2}ag^{1/2})$ for $a\in M$.
Since $\omega_g$ is normal, $\sigma$-strong convergence gives
$\omega_g(y_n^*y_n)\to0$. Hence H\"older's inequality yields
\[
 \begin{aligned}
 \|g^{1/2}y_ng^{1/2}\|_1
 &\le \|g^{1/2}\|_2\,\|y_ng^{1/2}\|_2\\
 &=\Tr(g)^{1/2}\omega_g(y_n^*y_n)^{1/2}
 \longrightarrow0.
 \end{aligned}
\]
For the assertion about adjoints, apply this argument to $y_n^*$
and use adjoint invariance of the $L^1$ norm.

For the tracial assertions use the same proof in $L^p(M,\tau)$;
see also \cite[Chapter IX, \S2]{TakesakiII}.
\end{proof}

\subsection{A fixed normal semifinite weight and unbounded majorants}
\label{ldct-merged:subsection}

The aim of this subsection is the positive-majorant `integration 
criterion' in Theorem~\ref{ldct-merged:criterion}. We first
construct the normal positive functionals associated with
integrable affiliated majorants, and then show that those
associated with bounded positive majorants already form a
norm-total family in the predual.

Let $\varphi$ be a faithful normal semifinite weight on $M$, let
\[
 \mathfrak n_\varphi=\{a\in M:\varphi(a^*a)<\infty\},
\]
and let $(H_\varphi,\pi_\varphi,\Lambda_\varphi)$ be its GNS construction.
Thus $\Lambda_\varphi(\mathfrak n_\varphi)$ is dense,
$\|\Lambda_\varphi(a)\|^2=\varphi(a^*a)$, and $\pi_\varphi$ is faithful
and normal. We use the linear extension of $\varphi$ to
$\mathfrak m_\varphi=\mathfrak n_\varphi^*\mathfrak n_\varphi$
in expressions such as $\varphi(a^*xa)$. Here products of these spaces mean the linear span of products.

\begin{definition}[The integral associated with an unbounded majorant]
\label{ldct-merged:integral-definition}
Let $g$ be a positive selfadjoint operator affiliated with $M$. Define the truncation $g_k=\min(g,k\cdot1),$ and let 
 $\varphi(g):=\sup_k\varphi(g_k).$ 
We call $g$ integrable for this construction if $\varphi(g)<\infty$.
For integrable $g$ and $x\in M$, define
\begin{equation}
 g^{\frac12}\varphi g^{\frac12}(x)
   =\lim_{k\to\infty}\varphi(g_k^{1/2}xg_k^{1/2}).
 \label{ldct-merged:regularized-integral}
\end{equation}
Existence of this limit, including norm convergence of the functionals,
follows from the next result.
\end{definition}

Related constructions of functionals associated with positive
affiliated operators, and their connections with normal
weights, are developed in \cite{Novikov} (see
e.g.\ Theorems 2 and 11, Corollaries 3 and 8, there). We give a direct
GNS proof of the finite-weight statement needed here,
including the norm estimate for truncations. 

\begin{lemma}[Construction of the regularized weighted integral]
\label{ldct-merged:integral-construction}
Suppose that $\varphi(g)<\infty$. The functional in \eqref{ldct-merged:regularized-integral} exists, is normal
and positive, and satisfies
$\|g^{\frac12}\varphi g^{\frac12}\|
 =g^{\frac12}\varphi g^{\frac12}(1)=\varphi(g)$.
More precisely, if
$\omega_k(x)=\varphi(g_k^{1/2}xg_k^{1/2})$, then
\begin{equation}
 \|g^{\frac12}\varphi g^{\frac12}-\omega_k\|
 \le 2\varphi(g)^{1/2}
        \bigl(\varphi(g)-\varphi(g_k)\bigr)^{1/2}.
 \label{ldct-merged:truncation-estimate}
\end{equation}
\end{lemma}

\begin{proof}
Put $\xi_k=\Lambda_\varphi(g_k^{1/2})$.
Since the truncations commute, for $l\ge k$ the scalar inequality
for square roots gives
\[
 (g_l^{1/2}-g_k^{1/2})^2\le g_l-g_k.
\]
Consequently
\[
 \|\xi_l-\xi_k\|^2\le\varphi(g_l)-\varphi(g_k).
\]
Thus $\xi_k\to\xi_g$ in $H_\varphi$, $\|\xi_g\|^2=\varphi(g)$, and
\[
 \omega_k(x)=\langle\pi_\varphi(x)\xi_k,\xi_k\rangle
 \longrightarrow\langle\pi_\varphi(x)\xi_g,\xi_g\rangle
 =g^{\frac12}\varphi g^{\frac12}(x).
\]
Normality follows from normality of $\pi_\varphi$.
Writing $\omega_\xi(x)=\langle\pi_\varphi(x)\xi,\xi\rangle$, the estimate
\[
 \|\omega_{\xi_g}-\omega_{\xi_k}\|
 \le(\|\xi_g\|+\|\xi_k\|)\|\xi_g-\xi_k\|
\]
proves \eqref{ldct-merged:truncation-estimate}.
\end{proof}

For brevity, we write $\omega_g=g^{\frac12}\varphi g^{\frac12}$, and call it the regularized weighted integral.

\begin{remark}[What the notation means]
\label{ldct-merged:notation}
One may denote the quantity in \eqref{ldct-merged:regularized-integral} by
$\int g^{1/2}xg^{1/2}\,d\varphi$, provided this is explicitly defined as
the regularized integral above. For a selfadjoint contraction $x$, the
quadratic form on $\operatorname{Dom}(g^{1/2})$ given by
\[
 q_{g,x}(\xi,\eta)=\langle xg^{1/2}\xi,g^{1/2}\eta\rangle
\]
satisfies $-q_g\le q_{g,x}\le q_g$.
Every bounded truncation satisfies
$-g_k\le g_k^{1/2}xg_k^{1/2}\le g_k$.
For a general weight we do not usually identify an arbitrary unbounded
formal product with a selfadjoint measurable operator.

For a trace and $g\in L^1(M,\tau)_+$, the regularized integral is
exactly $\tau(g^{1/2}xg^{1/2})$, by $L^2$ convergence of the square roots
and H\"older's inequality.
\end{remark}

We will use repeatedly a little-known  characterization of semifiniteness (that is, weak* density in $M_+$ of 
$\{ x \in M_+ : \varphi(x) < \infty \}$) of normal weights.  It may be deduced from \cite[Lemma 1.3]{HJKN} 
or its proof, however for completeness 
we include a different proof.   We are grateful to Brent Nelson for the latter reference. 

\begin{lemma}
\label{ldct-merged:finite-projections}
A weight $\varphi$ on a von Neumann algebra $M$ is semifinite if and only if there is a net of projections $(e_\alpha)$ in $M$ such that
$\varphi(e_\alpha)<\infty$ and $e_\alpha\to1$ strongly. \end{lemma}

\begin{proof}
Suppose first that $\varphi$ is semifinite. Let $\mathcal D=\{a\in M_+:\varphi(a)<\infty\}$. Direct $\mathcal D$ by the
usual operator order.  It  is directed because $a+b\in\mathcal D$ is an
upper bound for $a,b\in\mathcal D$.
Writing $s(a)$ for the support projection of $a$, we have
\[
\bigvee_{a\in\mathcal D}s(a)=1.
\]
Indeed, if $p$ denotes this supremum, then
$\operatorname{span}\mathcal D\subseteq pMp$. Since $pMp$ is weak$^*$
closed and $\operatorname{span}\mathcal D$ is weak$^*$ dense in $M$,
we obtain $pMp=M$, and hence $p=1$.

For $a\in\mathcal D$, set
\[
h_a=a(1+a)^{-1}=1-(1+a)^{-1}.
\]
Because inversion reverses the order on positive invertible operators,
$(h_a)$ is an increasing net of positive contractions. Thus
$h_a\to h$ strongly for some $0\le h\le1$.
For every $b\in\mathcal D$ and $t>0$, we have $tb\in\mathcal D$, so
\[
h\ge h_{tb}=tb(1+tb)^{-1}
\xrightarrow[t\to\infty]{\mathrm{strongly}}s(b).
\]
Consequently $h\ge s(b)$ for every $b\in\mathcal D$.
Since $h\le1$, this implies that $h$ acts as the identity on the range
of each $s(b)$. Their supremum is $1$, so $h=1$.

Now take
\[
e_a=\mathbf 1_{[1,\infty)}(a).
\]
Functional calculus gives
\[
e_a\le a,
\qquad
0\le1-e_a\le2(1+a)^{-1}=2(1-h_a).
\]
The first inequality yields
\[
\varphi(e_a)\le\varphi(a)<\infty.
\]
The second yields, for every vector $\xi$ in the Hilbert space on which
$M$ acts,
\[
\|(1-e_a)\xi\|^2
=\langle(1-e_a)\xi,\xi\rangle
\le2\langle(1-h_a)\xi,\xi\rangle
\longrightarrow0.
\]
Hence $e_a\to1$ strongly.

Conversely, suppose that $(e_\alpha)$ is a net of projections with
$\varphi(e_\alpha)<\infty$ and $e_\alpha\to1$ strongly.
For every $x\in M_+$,
\[
0\le e_\alpha x e_\alpha\le\|x\|e_\alpha,
\]
so $\varphi(e_\alpha x e_\alpha)<\infty$.
Moreover, $e_\alpha x e_\alpha\to x$ strongly and, since these
compressions are norm bounded, also weak$^*$.
Thus every element of $M_+$ belongs to the weak$^*$ closure of
$\mathcal D$. Since $M_+$ linearly spans $M$,
$\operatorname{span}\mathcal D$ is weak$^*$ dense in $M$.
Therefore $\varphi$ is semifinite.
\end{proof}

\begin{theorem}[Positive-majorant criterion]
\label{ldct-merged:criterion}
Let $\varphi$ be a faithful normal semifinite weight on $M$,
let $(x_n)$ be norm bounded in $M$, and let $x\in M$.
The associated bounded-majorant functionals satisfy
\[
 \overline{\operatorname{span}}^{\|\cdot\|}
 \bigl\{\omega_g:g\in M_+,\ \varphi(g)<\infty\bigr\}
 =M_*.
\]
Moreover, the following are equivalent:
\begin{enumerate}
\renewcommand{\theenumi}{\roman{enumi}}
\renewcommand{\labelenumi}{\textup{(\theenumi)}}
\item $x_n\to x$ weak*;
\item $\psi(x_n)\to\psi(x)$ for every sequentially normal state $\psi$ on $M$;
\item $\omega_g(x_n)\to\omega_g(x)$ for every bounded $g\in M_+$
      with $\varphi(g)<\infty$;
\item $\omega_g(x_n)\to\omega_g(x)$ for every positive selfadjoint
      affiliated $g$ with $\varphi(g)<\infty$;
\item $\varphi(a^*x_na)\to\varphi(a^*xa)$ for every
      $a\in\mathfrak n_\varphi$;
\item $\omega_g(b^*x_nb)\to\omega_g(b^*xb)$ for every positive
      selfadjoint affiliated $g$ with $\varphi(g)<\infty$
      and every $b\in M$.
\end{enumerate}
 Conditions \textup{(i)} and \textup{(iii)--(vi)} are also equivalent
for norm bounded nets.
\end{theorem}

\begin{proof}
Theorem \ref{main} gives
(i)$\Rightarrow$(ii). Normal states are sequentially normal and their
linear span is $M_*$, so (ii)$\Rightarrow$(i).
Normality of every $\omega_g$ gives (i)$\Rightarrow$(iv), and
(iv)$\Rightarrow$(iii) is immediate.

To prove (iii)$\Rightarrow$(i), we first prove separation, using only
bounded positive $g$ of finite weight.
Suppose $y=y^*\in M$ and $\omega_g(y)=0$ for every such $g$.
Fix any projection $e$ with $\varphi(e)<\infty$, and put $z=eye$.
For all real $t$ in a sufficiently small interval about zero,
$a_t=e+tz$ is positive in $eMe$. Put $g_t=a_t^2$. Then
\[
 0\le g_t\le\|a_t\|^2e,\qquad
 \varphi(g_t)<\infty,\qquad
 g_t^{1/2}=a_t.
\]
The restriction of $\varphi$ to $eMe$ extends to  a finite positive functional on $eMe$,
and therefore
\[
 0=\omega_{g_t}(y)=\varphi\bigl((e+tz)y(e+tz)\bigr)
\]
is a polynomial identity in $t$.   All terms here lie in $eMe$. Its linear coefficient is
\[
 \varphi(zye+eyz)=2\varphi(z^2),
\]
because $z=eye$. Faithfulness gives $z^2=0$ and $z=0$.
Applying this to the projections in Lemma~\ref{ldct-merged:finite-projections}
and using $e_\alpha ye_\alpha\to y$ strongly shows that $y=0$.

If an arbitrary $y\in M$ is annihilated by all these $\omega_g$, positivity
of the functionals implies that both $\operatorname{Re}y$ and
$\operatorname{Im}y$ are annihilated. Thus $y=0$.
The complex linear span $V$ of these normal functionals is consequently
norm dense in $M_*$: otherwise Hahn--Banach would give a nonzero element
of $(M_*)^*=M$ annihilating $V$.

Now assume (iii) and put $C=\sup_n\|x_n-x\|<\infty$.
For $\rho\in M_*$ and $\varepsilon>0$, choose $\eta\in V$ with
$\|\rho-\eta\|<\varepsilon$. Condition (iii) gives $\eta(x_n-x)\to0$, so
\[
 \limsup_n|\rho(x_n-x)|\le C\varepsilon.
\]
Letting $\varepsilon\downarrow0$ proves (i).

For $a\in\mathfrak n_\varphi$, the GNS formula
\[
 \varphi(a^*ya)
 =\langle\pi_\varphi(y)\Lambda_\varphi(a),\Lambda_\varphi(a)\rangle
 \qquad(y\in M)
\]
shows that the functional in (v) is normal.
The functionals in (vi) are normal by
Lemma~\ref{ldct-merged:integral-construction}.
Thus (i) implies (v) and (vi).
Condition (v) implies (iii) by taking $a=g^{1/2}$,
while (vi) implies (iv) by taking $b=1$.
The same arguments, excluding the condition involving sequentially normal
states, apply to norm bounded nets.
\end{proof}

 \begin{corollary}[Strong convergence and weighted domination]
\label{ldct-merged:strong-convergence}
Suppose that $(x_n)$ is norm bounded and $x_n\to x$
$\sigma$-strongly. For every positive affiliated $g$ with
$\varphi(g)<\infty$, write
$\omega_g=g^{1/2}\varphi g^{1/2}$ as above. Then
\[
 \omega_g(|x_n-x|)\longrightarrow0,
 \qquad
 \omega_g((x_n-x)^*(x_n-x))\longrightarrow0.
\]
In the commutative case the first conclusion is precisely
$\int |x_n-x|g\,d\mu\to0$.
\end{corollary}
\begin{proof}
Put $y_n=x_n-x$. The functional $\omega_g$ is normal and positive
by Lemma~\ref{ldct-merged:integral-construction}.
Thus $\sigma$-strong convergence gives $\omega_g(y_n^*y_n)\to0$.
Cauchy--Schwarz gives
\[
 \omega_g(|y_n|)^2
 \le \omega_g(1)\,\omega_g(|y_n|^2)
 =\varphi(g)\,\omega_g(y_n^*y_n)\longrightarrow0.
\] This is what was needed. 
\end{proof}

\noindent\textbf{Remark.}
We remark on the classical meaning of the above.
In the commutative setting, $\varphi(g)=\int g\,d\mu$ by monotone
convergence. Moreover, directly from the estimate
\[
 \left|\int(g-g_k)x\,d\mu\right|
 \le\|x\|_\infty\int(g-g_k)\,d\mu\longrightarrow0,
\]
one has $(g^{\frac12}\varphi g^{\frac12})(x)=\int gx\,d\mu$.

Corollary~\ref{ldct-merged:strong-convergence} gives the weighted
absolute-error conclusion. In the commutative case the normalization
in Lemma~\ref{classical} gives $\int|f_n-f|\,d\mu\to0$.

\section{Almost uniform convergence for a faithful normal semifinite weight}
\label{au-weight-section}

Noncommutative counterparts of almost-everywhere convergence
have been studied in many papers.  Indeed there are  several different notions of such convergence.
For example, Hensz, Jajte and Paszkiewicz developed bundle convergence
in von Neumann algebras and their $L^2$-spaces \cite{HJP};
Le Gac and M\'oricz proved Beppo Levi and dominated
convergence theorems in that setting \cite{LGM}.
Budde, Labuschagne and Steyn develop notions applicable
also to nontracial Haagerup $L^p$-spaces \cite{BLS}.
Here we study projection-cutoff conditions for norm bounded
sequences in $M$, with particular emphasis on local
conditions relative to a faithful normal semifinite weight
and their consequences for integration. 

Throughout this section, $M$ is a von Neumann algebra and $\varphi$ is a
faithful normal semifinite weight (f.n.s.w.) on $M$. 
 
A \emph{cutoff} here means a projection $e$ that cuts down the algebra, and such that $1-e$ has small weight.
We use these to study
$(x_n-x)e$, or the bilateral compression $e(x_n-x)e$, on the retained
part of the algebra `supported by'  $e$, while $1-e$ has small weight. In a corner $pMp$
we take $e\le p$ and require $\varphi(p-e)$ to be small. Classically,
$e=\mathbf1_E$ restricts the functions to a large measurable set $E$.
The one-sided a.u.\ definition originates in Lance's work
\cite{Lance}, for a faithful normal state; below we also use
semifinite weights and the bilateral variant.

\subsection{Global almost uniform convergence}

\begin{proposition}[Almost-uniform bounded convergence]
\label{prop:almost-uniform-convergence}
Let $\psi$ be a sequentially normal state, let $(x_n)$ be norm
bounded, and let $x\in M$.  Suppose that there are projections
$p_m\uparrow1$ such that, for every fixed $m$,
\begin{equation}\label{eq:almost-uniform-compression}
                    \|(x_n-x)p_m\|\longrightarrow0.
\end{equation}
Then $\psi(x_n)\to\psi(x)$.
\end{proposition}
\begin{proof}
Again put $y_n=x_n-x$ and $C=\sup_n\|y_n\|$.  For fixed $m$,
Cauchy--Schwarz gives
\begin{align*}
 |\psi(y_n)|
 &\leq |\psi(y_np_m)|+|\psi(y_n(1-p_m))|\\
 &\leq \|y_np_m\|+C\psi(1-p_m)^{1/2}.
\end{align*}
First let $n\to\infty$.  The first term tends to zero by
\eqref{eq:almost-uniform-compression}.  Then let $m\to\infty$.
Sequential normality gives $\psi(1-p_m)\downarrow0$, proving the
claim.
\end{proof}

\begin{remark}[Relation with noncommutative Egoroff theorems]
\label{rem:egoroff-comparison}
Noncommutative Lusin--Egoroff approximation results give norm
approximation after compression, with the exceptional projection
small for a prescribed normal positive functional; see
\cite[Theorem III.4.13, p.~85]{Tak1} and \cite{ABB}.

To apply the estimate in
Proposition~\ref{prop:almost-uniform-convergence} to a sequentially
normal state $\psi$, the exceptional projections must be small
for $\psi$ itself. In that proposition, $p_m\uparrow1$ ensures
$\psi(1-p_m)\downarrow0$ directly by sequential normality.
Smallness of $\omega(1-p)$ for a single prescribed normal positive
functional $\omega$ does not by itself imply smallness of
$\psi(1-p)$. Alternatively, if the exceptional projections
converge weak* to zero, Theorem~\ref{main} gives their convergence
to zero under $\psi$. Noncommutative Egoroff approximation also enters the proof of \cite[Theorem 4.3(ii)]{CHW}, which identifies \(\sigma\)-strong sequential convergence on \(M_{\mathrm{sa}}\) with convergence in the sequential order topology when \(M\) is \(\sigma\)-finite.
\end{remark}

All operators in
the sequences below belong to $M$.

\begin{definition}[Almost uniform convergence relative to a weight]\label{cutoff}
We say that $x_n\to x$ \emph{almost uniformly relative to $\varphi$} (or $\varphi$-a.u.\ for short) 
if, for every $\varepsilon>0$, there is a projection $e\in M$ such that
\[
 \varphi(1-e)<\varepsilon,
 \qquad \|(x_n-x)e\|\longrightarrow0.
\]
We say that $x_n\to x$ \emph{bilaterally almost uniformly relative to
$\varphi$} if, for every $\varepsilon>0$, there is a projection $e\in M$
such that
\[
 \varphi(1-e)<\varepsilon,
 \qquad \|e(x_n-x)e\|\longrightarrow0.
\]
We abbreviate these conditions as $\varphi$-a.u.\ and $\varphi$-b.a.u.,
respectively. The projection is fixed while $n$ tends to infinity.
Clearly $\varphi$-a.u.\ convergence implies $\varphi$-b.a.u.\ convergence.
\end{definition}

\begin{lemma}[Small-weight projections converge strongly to zero]
\label{small-weight-projections}
If $p_k\in M$ are projections and $\varphi(p_k)\to0$, then $p_k\to0$
strongly in every faithful normal representation of $M$.
\end{lemma}
\begin{proof}
For weights on $M_+$, normality is known to be equivalent to weak* lower semicontinuity on $M_+$ \cite{Haag}.
The positive unit ball of $M$ is weak* compact. If a subnet
$p_{k_\alpha}$ converges weak* to $p$, then $p\ge0$ and
\[
 \varphi(p)\le\liminf_\alpha\varphi(p_{k_\alpha})=0.
\]
Faithfulness gives $p=0$. 
Thus zero is the only weak* accumulation point, and so by compactness we have 
$p_k\uw0$. In a faithful normal
representation, for every vector $\xi$,
\[
 \|p_k\xi\|^2=\langle p_k\xi,\xi\rangle\longrightarrow0.
\]
That is, we have  strong convergence to 0. 
\end{proof}

In particular,
$\varphi(1-e_k)\to0$ implies $e_k\to1$ strongly.

\begin{corollary}[Bounded spectral and quadratic-mean convergence]
\label{spectral-l2-strong}
Let $\varphi$ be a faithful normal semifinite weight on $M$,
and let $(y_n)\subset M$ satisfy $\sup_n\|y_n\|<\infty$.
Suppose that
\[
 \varphi\!\left(1_{(\varepsilon,\infty)}(|y_n|)\right)
 \longrightarrow0
 \qquad\text{for every }\varepsilon>0.
\]
Then $y_n\to0$ strongly in every faithful normal representation,
and hence weak*.

In particular, the same conclusions hold if
\[
 \varphi(y_n^*y_n)\longrightarrow0.
\]
Under either hypothesis,
\[
 \psi(y_n)\longrightarrow0
 \quad\text{for every sequentially normal state }\psi\text{ on }M.
\]
\end{corollary}

\begin{proof}
By Lemma~\ref{small-weight-projections}, the spectral projections
converge weak* to zero. Then  Lemma~\ref{prop:projection-measure-convergence}
and Theorem~\ref{main} give the first assertion. The other then follows since    $\varepsilon^2\one_{(\varepsilon,\infty)}(|y_n|)\le y_n^*y_n$.
\end{proof}

\begin{theorem}[Bounded a.u.\ and b.a.u.\ convergence imply weak* convergence]
\label{bau-weak*}
Suppose $\sup_n\|x_n\|<\infty$ and $x\in M$. If $x_n\to x$
$\varphi$-b.a.u., then $x_n\uw x$. Consequently, the same conclusion
holds for $\varphi$-a.u.\ convergence.
\end{theorem}
\begin{proof} Put $y_n=x_n-x$ and $C=\sup_n\|y_n\|$. 
Choose global cutoffs $e_k$ with $\varphi(1-e_k)<2^{-k}$ and $\| e_ky_ne_k \| \to 0$ with $n$ for fixed
$k$. 
Lemma~\ref{small-weight-projections} gives $e_k\to1$ strongly.
 For a normal state $\rho$,
Cauchy--Schwarz applied to
$y_n=e_ky_ne_k+(1-e_k)y_n+e_ky_n(1-e_k)$ gives
\[
 |\rho(y_n)|\le\|e_ky_ne_k\|+2C\rho(1-e_k)^{1/2}.
\]
First let $n\to\infty$, then $k\to\infty$. Normal positive
functionals span $M_*$, proving weak* convergence.
\end{proof}

\begin{corollary}[Application to the dominated convergence formulations]
\label{bau-ldct}
As before suppose that $\sup_n\|x_n\|<\infty$ and $x\in M$. If $x_n\to x$ globally or locally $\varphi$-b.a.u., then
\[
 \psi(x_n)\longrightarrow\psi(x)
 \quad\text{for every sequentially normal state }\psi\text{ on }M.
\]
Also, for every positive affiliated $g$ with $\varphi(g)<\infty$ and
every $b\in M$,
\[
 \Omb(b^*x_nb)\longrightarrow\Omb(b^*xb).
\]
Under the same hypotheses one also has
\[
 \|h^{1/2}(x_n-x)h^{1/2}\|_1\longrightarrow0
 \qquad(h\in L^1(M)_+),
\]
where $L^1(M)$ is the Haagerup space; see
Theorem~{\rm \ref{ldct-addon:theorem}(i)}. \end{corollary}
\begin{proof} Apply Theorem~\ref{bau-weak*} or Theorem~\ref{ldct-addon:theorem}(i), followed by
Theorem~\ref{main}, for the first assertion.
The second follows from the normality of
$a\mapsto\Omb(b^*ab)$, established in
Lemma~\ref{ldct-merged:integral-construction}.
\end{proof}

Further comparisons with the spectral formulations in
\cite{BLS} are given in the subsection entitled ``Comparison with spectral convergence'' below.

\subsection{Almost uniform convergence and dominated convergence in the predual}
\label{predual:subsection}

The first principal result of this subsection is Theorem~\ref{predual:dct}:
almost uniform convergence in the predual, together with domination by a
fixed positive normal functional, implies convergence in the predual norm.
The reference state used to define convergence may be chosen independently
of the dominating functional. In the commutative case the convergence
condition is exactly a.e.\ convergence. We first give the theorem directly
in the ordered predual $M_*$.  We use this to give a nc LDCT in the Haagerup $L^1$-space, this is 
the second principal result of this subsection.  We will use the Haagerup realization  of $L^1$ to
connect it with the bounded convergence result of Corollary~\ref{bau-ldct}.

For related order-topology results in the predual, see
\cite[Theorem 5]{CH}. 
That theorem identifies the norm and order topologies on
$M_*^{\sa}$. In our proof below, the estimate
\eqref{predual:estimate} establishes norm convergence directly
from domination and almost uniform convergence.

For a projection $e\in M$ and $\phi\in M_*$, write
\[
 (e\phi e)(x)=\phi(exe)\qquad(x\in M).
\]
We say that $\phi$ is supported by $p$ if $\phi=p\phi p$.
Fix a nonzero $\sigma$-finite projection $p$ and a faithful normal state
on $N=pMp$, whose identity is $p$. Extend the state to $M$ by compression
and denote the extension by $\psi$; thus
\[
 \psi(x)=\psi(pxp),\qquad \psi(1)=\psi(p)=1,\qquad s(\psi)=p.
\]
The same convention applies to other functionals on the corner. Extension
by compression identifies $N_*$ isometrically and order isomorphically
with the functionals in $M_*$ supported by $p$: compression is contractive
and positive, and restriction to $N$ is its inverse on this range. 
The existence of suitable supporting corners is explained
after Theorem~\ref{predual:dct}.

\begin{definition}[Bilateral almost uniform convergence in the predual]
\label{predual:definition}
Let selfadjoint $\phi_n,\phi\in M_*$ be supported by $p$. Write
$\phi_n\baupre{\psi}\phi$ if, for every $\varepsilon>0$, there are a
projection $e\le p$ and numbers $\delta_n\ge0$ tending to zero such that
\begin{equation}\label{predual:cutoff}
 \psi(p-e)<\varepsilon,\qquad
 -\delta_n  \, e\psi e\le e(\phi_n-\phi)e\le\delta_n \, e\psi e
\end{equation}
for all sufficiently large $n$. The projection $e$ is independent of $n$.
\end{definition}

The order inequalities in \eqref{predual:cutoff} are equivalently
\begin{equation}\label{predual:tests}
 |(\phi_n-\phi)(x)|\le\delta_n\psi(x)
 \qquad(x\in eM_+e).
\end{equation}
Indeed, 
every positive element of
$eMe$ is the compression of a positive element of $M$. On $N_*$ this is Petz's almost uniform convergence; see
\cite[pp.~248--249]{Petz} and the definition recalled in
\cite[p.~1]{Bartoszewicz}. We use $\mathrm{b.a.u.}_*$ to distinguish
the bilateral compression of functionals from convergence of bounded
elements of $M$. Petz also proves that every norm-convergent sequence in
$N_*^{\mathrm{sa}}$ has a subsequence converging almost uniformly,
relative to the fixed faithful normal state, to the same limit
\cite[Corollary, p.~250]{Petz}. 

\begin{theorem}[Dominated convergence in the predual]
\label{predual:dct}
Suppose selfadjoint $\phi_n,\phi\in M_*$ are supported by $p$ and
$\phi_n\baupre{\psi}\phi$. Let $\omega\in M_*^+$ be supported by $p$,
and suppose that
\[
 -C\omega\le\phi_n\le C\omega\qquad(n\ge1)
\]
for some $0\le C<\infty$. Then
\[
 -C\omega\le\phi\le C\omega,\qquad
 \|\phi_n-\phi\|\longrightarrow0,\qquad
 \phi_n(1)\longrightarrow\phi(1).
\]
\end{theorem}

\begin{proof}
Put $\eta_n=\phi_n-\phi$ and
$\Omega=C\omega+|\phi|$, where $|\phi|=\phi^++\phi^-$ is the sum of
the Jordan parts of $\phi$. These parts are supported by $p$, by the
order isomorphism $N_*\cong pM_*p$. Since
$-|\phi|\le\phi\le|\phi|$, we have $-\Omega\le\eta_n\le\Omega$.
Thus
\[
 u_n=\frac{\Omega+\eta_n}{2},\qquad
 v_n=\frac{\Omega-\eta_n}{2}
\]
are positive normal functionals, with $u_n-v_n=\eta_n$ and
$u_n+v_n=\Omega$.

Fix a cutoff $e$ and constants $\delta_n$ as in
\eqref{predual:cutoff}, and put $q=p-e$.
The positive functionals
$(\delta_n e\psi e+e\eta_ne)/2$ and
$(\delta_n e\psi e-e\eta_ne)/2$ have difference $e\eta_ne$ and
sum $\delta_n e\psi e$. The triangle inequality and the formula
$\|\theta\|=\theta(p)$ for $\theta\in N_*^+$ therefore give
\begin{equation}\label{predual:compressed-norm}
 \|e\eta_ne\|\le\delta_n\psi(e).
\end{equation}
For $x\in N$ with $\|x\|\le1$, Cauchy--Schwarz for $u_n,v_n$,
followed by scalar Cauchy--Schwarz, gives
\begin{align*}
 |\eta_n(qx)|
 &\le \sqrt{u_n(q)u_n(x^*x)}+\sqrt{v_n(q)v_n(x^*x)}\\
 &\le \sqrt{\Omega(q)\Omega(x^*x)}
 \le \sqrt{\Omega(q)\Omega(p)}.
\end{align*}
Similarly,
$|\eta_n(exq)|\le\sqrt{\Omega(exx^*e)\Omega(q)}
\le\sqrt{\Omega(p)\Omega(q)}$.
Using $x=exe+qx+exq$, and the equality of the norms in $N_*$ and
$M_*$, we obtain
\begin{equation}\label{predual:estimate}
 \|\eta_n\|\le\delta_n\psi(e)
       +2\sqrt{\Omega(p)\Omega(p-e)}.
\end{equation}

For each $k$, choose $e_k\le p$ from Definition \ref{predual:definition} with
$\psi(p-e_k)<2^{-k}$, and write $\delta_{k,n}\to0$ for its constants.
We have $p-e_k\to0$ weak*: every weak* limit point in the positive
unit ball of $N$ is positive and has $\psi$-value zero, hence is zero
by faithfulness.  Weak* compactness then gives convergence.
Normality of $\Omega$ yields $\Omega(p-e_k)\to0$.
First let $n\to\infty$ in \eqref{predual:estimate} with $k$ fixed,
and then let $k\to\infty$. This proves $\|\eta_n\|\to0$.
The positive cone of $M_*$ is norm closed, so the inequalities
$C\omega\pm\phi_n\ge0$ pass to the limit and give
$-C\omega\le\phi\le C\omega$.
Finally, $|\phi_n(1)-\phi(1)|\le\|\eta_n\|\to0$.
\end{proof}

Theorem~\ref{predual:dct} can also be deduced from Petz's
transfer theorem \cite[Theorem~2, p.~250]{Petz}, together with
normalization as in Proposition~\ref{predual:normalization}
and the bounded compression estimate
\eqref{ldct-addon:estimate} below. For this purpose one uses
the faithful normal state proportional to
$\psi+C\omega+|\phi|$: the assumed predual almost uniform
convergence also holds relative to this state, and the errors
$\phi_n-\phi$ are uniformly order dominated by it.
The proof above gives a direct argument in the ordered predual.

For complex functionals one may apply the theorem separately to real and
imaginary parts whenever each part satisfies its domination and convergence
hypotheses. Then
$\|\phi_n-\phi\|\le\|\operatorname{Re}(\phi_n-\phi)\|
+\|\operatorname{Im}(\phi_n-\phi)\|\to0$.

In the classical case of $M=L^\infty(X,\mu)$ with $\mu$ localizable one can use Egoroff to show that this predual convergence is equivalent to a.e.\ convergence, and that  the last result is
a variant of the classical LDCT.  Indeed write $p=1_E, d\psi=r \, d\mu$, and $d\varphi_n=f_n\,d\mu, d\varphi=f \, d\mu$. Faithfulness on $pMp$ means that $r>0$ almost everywhere on $E$. Definition \ref{predual:definition} is precisely almost uniform convergence of $(f_n-f)/r$ to zero for the probability measure $r \, d\mu$. By Egoroff’s theorem and the converse implication from almost uniform to almost everywhere convergence, this is equivalent to $f_n\to f$ $\mu$-almost everywhere on $E$.

\smallskip

\noindent\emph{Supporting corners.}
Recall that if $0\ne\omega\in M_*^+$ and $p=s(\omega)$, then $pMp$
is $\sigma$-finite and $\omega/\omega(1)$ is a faithful normal state
on $pMp$. Moreover, $-C\omega\le\phi\le C\omega$ implies
$\phi=p\phi p$, so $p$  supports every functional dominated
in this way by $\omega$.

\smallskip
\noindent\emph{Normalization in the Haagerup realization of $L^1$.}
We now specialize the reference state to the normalized majorant and
compare the predual theorem with bounded convergence.
Throughout the remainder of the subsection fix
$0\ne g\in L^1(M)_+$ in the Haagerup space of Subsection~\ref{HL1}, and set
\begin{equation}\label{predual:haagerup-notation}
 \omega_h(x)=\Tr(hx),\qquad m=\Tr(g),\qquad
 p=s(g)=s(\omega_g),\qquad N=pMp.
\end{equation}
Here $\omega_h$ is defined for every $h\in L^1(M)$, and $\Tr$ is the
canonical integral on $L^1(M)$, not the semifinite trace on the crossed
product. Put
\begin{equation}\label{predual:reference-state}
 \psi_g(x)=m^{-1}\omega_g(x)\qquad(x\in M).
\end{equation}
Its restriction to $N$ is faithful, and
$\psi_g(a)=m^{-1}\Tr(g^{1/2}ag^{1/2})$ for $a\in N$. If $M$ has a faithful normal semifinite trace $\tau$, its restriction
to $N$ is again such a trace. 
Indeed finite-trace approximants below a positive element of $N$ given by semifiniteness in $M$ already lie in $N$, since they are supported by $p$.

We use the standard isometric order identification
$L^1(N)=pL^1(M)p$; see \cite[Section~1.2 and Remark~3.2]{HJX}.
By the support observation above, $-Cg\le f\le Cg$ implies $f=pfp$,
so every dominated density lives in  the appropriate corner.

\begin{proposition}[Normalization by a fixed majorant]
\label{predual:normalization}
The map $J_g:N\to L^1(M)$, $J_g(a)=g^{1/2}ag^{1/2}$, is injective,
positive, linear, and adjoint preserving, with $\|J_g(a)\|_1\le m\|a\|$.
For every $C\ge0$ it maps $[-Cp,Cp]_N$ bijectively onto
$[-Cg,Cg]_{L^1(M)}$.
Equivalently, $-Cg\le f\le Cg$ precisely when
$f=g^{1/2}ag^{1/2}$ for a unique selfadjoint $a\in N$ with $\|a\|\le C$.
On bounded subsets $J_g$  is a homeomorphism from the weak* topology to
the relative weak $L^1(M)$ topology. In particular, for a norm-bounded
sequence $(a_n)$ and $a\in N$,
\begin{equation}\label{predual:weak-equivalence}
 a_n  \longrightarrow a \; {\rm weak*} 
 \quad\Longleftrightarrow\quad
 J_g(a_n)\longrightarrow J_g(a)\ \text{weakly in }L^1(M).
\end{equation}
Write $I_g=J_g(N)$ and $D_g=J_g^{-1}:I_g\to N$.
\end{proposition}

\begin{proof}
H\"older's inequality and the standard $L^p$ product operations give
$\|J_g(a)\|_1\le m\|a\|$, linearity, positivity, and preservation of
adjoints; see \cite[Section~1.2]{HJX}.
Put $\xi=g^{1/2}\in L^2(N)$.  Of course we use the inner product 
$\langle u,v\rangle=\Tr(v^*u)$.
In the Haagerup standard form, left multiplication $\lambda(x)u=xu$
and right multiplication $\rho(x)u=ux$ are faithful normal actions
satisfying $\lambda(N)'=\rho(N)$; $\rho$ is a $*$-anti-representation.
Since $\omega_g$ is faithful, $\xi$ is separating and cyclic, so
$N\xi$ is dense; see \cite[Theorem~11.30 and Lemma~3.9(1)]{Hiai}.
For $a,x,y\in N$, cyclicity of the pairing gives
\begin{equation}\label{predual:matrix-coefficient}
 \langle\rho(a)x\xi,y\xi\rangle
 =\Tr(\xi y^*x\xi a)=\Tr(J_g(a)y^*x).
\end{equation}
Thus $J_g(a)=0$ implies $\rho(a)=0$, hence $a=0$.

If $0\le h\le g$, the bounded Radon--Nikodym theorem
\cite[Theorem~1.4.2]{Arveson}, applied to
$0\le\omega_h\le\omega_g$, gives $T\in\lambda(N)'$ with
$0\le T\le I$ and
$\omega_h(x)=\langle T\lambda(x)\xi,\xi\rangle$ for $x\in N$.
Write $T=\rho(b)$, where $0\le b\le p$.
Equation~\eqref{predual:matrix-coefficient}, with $y=p$, gives
$\omega_h=\omega_{J_g(b)}$, so $h=J_g(b)$ by the predual identification.
Together with positivity and $J_g(p)=g$, this proves
$J_g([0,p])=[0,g]$.
For $C>0$ and $-Cg\le f\le Cg$, put $h=(f+Cg)/(2C)$,
write $h=J_g(b)$ with $0\le b\le p$, and set $a=2Cb-Cp$.
Then $f=J_g(a)$ and $-Cp\le a\le Cp$, equivalently $a=a^*$ and
$\|a\|\le C$. Conversely, positivity applied to $Cp\pm a$ gives
$Cg\pm J_g(a)\ge0$. Uniqueness follows from injectivity of $J_g$.
For $C=0$ both intervals consist of zero.

Finally, for $z\in M$,
\[
 \Tr(J_g(a)z)=\Tr\bigl(a\,g^{1/2}(pzp)g^{1/2}\bigr).
\]
The functional  on the right side here belongs to $L^1(N)=N_*$, so $J_g$ is
weak*-to-weak continuous. Its restriction to each closed norm ball
is a continuous injection from a weak* compact space into a Hausdorff
space, hence a homeomorphism onto its image. This proves the remaining
assertions.
\end{proof}

In these coordinates, $\omega_{f_n}\baupre{\psi_g}\omega_f$ means
that, for every $\varepsilon>0$, there are a projection $e\le p$ and
$c_n\ge0$ tending to zero such that
\begin{equation}\label{predual:density-cutoff}
 \omega_g(p-e)<\varepsilon m,\qquad
 -c_n ege\le e(f_n-f)e\le c_n ege
\end{equation}
eventually in $n$.
Indeed, for $h\in L^1(M)$ and $x\in M$, the tracial property of the
Haagerup pairing \cite[Section~1.2]{HJX} gives
\[
 \omega_{ehe}(x)=\Tr(ehex)=\Tr(hexe)
 =\omega_h(exe)=(e\omega_he)(x).
\]
Thus $\omega_{ehe}=e\omega_he$. Since the density of $\psi_g$ is
$g/m$, we have $c_n=\delta_n/m$.
This is Definition~\ref{predual:definition} expressed in densities.

For $a_n,a\in N$, write $a_n\bau{\psi_g}a$ if, for every
$\varepsilon>0$, there is a projection $r\in N$ such that
\[
 \psi_g(p-r)<\varepsilon,\qquad \|r(a_n-a)r\|\longrightarrow0.
\]
This is the special case of $\varphi$-b.a.u.\ convergence defined in Definition \ref{cutoff},
applied to the algebra $N$, whose identity is $p$, with
$\varphi=\psi_g|_N$. The subscript $*$ distinguishes the predual
convergence of Definition~\ref{predual:definition}.

The qualitative predual-to-algebra b.a.u.\ implication in the
following lemma is Petz's transfer theorem
\cite[Theorem~2, p.~250]{Petz}, expressed in Haagerup language.
Petz states this implication and refers its proof to his 1985
preprint. We give an explicit projection construction and prove
the accompanying mass and order estimates.

\begin{lemma} \label{predual:transfer}
For each projection $e\in M$, let $r=s(g^{1/2}eg^{1/2})\in M$.
Then $r\le p$ and
\begin{equation}\label{predual:mass-transfer}
 \omega_g(p-r)\le\omega_g(1-e).
\end{equation}
For every $b=b^*\in N$ and every $\delta\ge0$,
\begin{equation}\label{predual:order-transfer}
 -\delta ege\le eJ_g(b)e\le\delta ege
 \quad\Longleftrightarrow\quad \|rbr\|\le\delta.
\end{equation}
In particular, for selfadjoint $f_n,f\in I_g$, predual convergence
implies b.a.u.\ convergence of the coefficients in $N$:
\[
 \omega_{f_n}\baupre{\psi_g}\omega_f
 \quad\Longrightarrow\quad
 D_g(f_n)\bau{\psi_g}D_g(f).
\] 
\end{lemma}

\begin{proof}
Put $\xi=g^{1/2}\in L^2(M)$ and take the polar decomposition
\[
 z=\xi e=vk^{1/2},\qquad k=z^*z=ege\in L^1(M)_+.
\]
It is known (see Section 1.2 in \cite{HJX}) that the polar partial isometry of a Haagerup $L^2(M)$ element belongs to
$M$. 
Thus
\[
 t=v^*v=s(k),\qquad r=vv^*=s(zz^*)=s(\xi e\xi)
\]
are projections in $M$. Since $pz=z$, its left support $r$ is below $p$.
For $b=b^*\in N$, associativity of the $L^p$ products gives
\begin{equation}\label{predual:polar-sandwich}
 eJ_g(b)e=z^*bz=k^{1/2}(v^*bv)k^{1/2},
 \qquad v^*bv\in(tMt)_{\sa}.
\end{equation}
Suppose first that $k\ne0$.
Proposition~\ref{predual:normalization}, applied with $k$ in place of
$g$, says that $J_k(c)=k^{1/2}ck^{1/2}$ is injective on $tMt$ and maps
$[-\delta t,\delta t]$ onto $[-\delta k,\delta k]$.
Consequently,
\[
 -\delta k\le J_k(v^*bv)\le\delta k
 \quad\Longleftrightarrow\quad
 -\delta t\le v^*bv\le\delta t.
\]
The forward implication uses surjectivity to find an element  in the
stated interval with the same image, and injectivity to identify it
with $v^*bv$. The reverse implication uses positivity of $J_k$.
Since $v^*bv$ is selfadjoint, the last bounds are equivalent to
$\|v^*bv\|\le\delta$.
The maps $c\mapsto vcv^*$ and $d\mapsto v^*dv$ are inverse
$*$-isomorphisms between $tMt$ and $rMr$, so they preserve the norm.
As $v(v^*bv)v^*=rbr$, this proves \eqref{predual:order-transfer}.
If $k=0$, then $z=0$ and $r=s(z z^*) = 0$, and both sides hold automatically.
Note that $J_g(b)e =0$ since $g^{1/2} e = z = 0$.

To prove \eqref{predual:mass-transfer}, $rz=z$ gives $(p-r)\xi e=0$, hence
$(p-r)\xi=(p-r)\xi(1-e)$.
Left multiplication by $p-r$ is contractive on $L^2(M)$, and cyclicity
of the pairing gives
\begin{align*}
 \omega_g(p-r)
 &=\Tr\bigl(\xi(p-r)\xi\bigr)=\|(p-r)\xi\|_2^2\\
 &=\|(p-r)\xi(1-e)\|_2^2
 \le\|\xi(1-e)\|_2^2\\
 &=\Tr\bigl((1-e)g(1-e)\bigr)=\omega_g(1-e).
\end{align*}
This proves \eqref{predual:mass-transfer}.
Finally, set $a_n=D_g(f_n)$ and $a=D_g(f)$; these are selfadjoint by
injectivity and preservation of adjoints of $J_g$.
Given $\varepsilon>0$, choose $e,c_n$ as in
\eqref{predual:density-cutoff}, and put $r=s(g^{1/2}eg^{1/2})$.
Then $r\le p$, and $r$ is independent of $n$.
Since $e\le p=s(g)$ and $\omega_g(1-p)=0$, we have
$\omega_g(1-e)=\omega_g(p-e)$.
Applying \eqref{predual:mass-transfer} and \eqref{predual:order-transfer},
the latter with $b=a_n-a$ and $\delta=c_n$, gives
\[
 \psi_g(p-r)<\varepsilon,\qquad
 \|r(a_n-a)r\|\le c_n\longrightarrow0,
\]
which are exactly the conditions for $\psi_g$-b.a.u.\ convergence in $N$.
\end{proof}

\begin{corollary}[Comparison with bounded convergence]
\label{predual:comparison}
Let $f_n,f\in L^1(M)_{\sa}$ and suppose that, for some $0\le C<\infty$,
\[
 -Cg\le f_n\le Cg,\qquad
 \omega_{f_n}\baupre{\psi_g}\omega_f.
\]
Then $-Cg\le f\le Cg$. The coefficients
$a_n=D_g(f_n)$ and $a=D_g(f)$ are selfadjoint, have norm at most $C$,
and satisfy
\[
 a_n\bau{\psi_g}a,\qquad a_n  \longrightarrow a \; {\rm weak*}.
\]
Moreover,
\[
 \|f_n-f\|_1\longrightarrow0,\qquad
 \Tr(f_n)\longrightarrow\Tr(f).
\]
\end{corollary}

\begin{proof} Since $\omega_f \leq \omega_h$ if and only if $f \leq h$, and $\| h \|_1 = \omega_h(1) = \| \omega_h \|$ if $h \geq 0$, we see that 
Theorem~\ref{predual:dct}, with majorant $\omega_g$, already gives
the limit's order bounds and the norm and integral conclusions.
For the connection with bounded convergence, use
Proposition~\ref{predual:normalization} to obtain
\[
 f_n=J_g(a_n),\qquad f=J_g(a),\qquad
 a_n,a\in N_{\sa},\qquad \|a_n\|,\|a\|\le C.
\]
Lemma~\ref{predual:transfer} converts the assumed predual
$\mathrm{b.a.u.}_*$ convergence into $a_n\bau{\psi_g}a$ in $N$.
Since $\psi_g|_N$ is a faithful normal state, hence a faithful normal
semifinite weight, Corollary~\ref{bau-ldct} applies to the norm-bounded sequence
$(a_n)$ in $N$ with $\varphi=\psi_g|_N$, and gives $a_n \to a$ weak*.
\end{proof}

\subsection*{Recovery of the classical bounded state formulation and LDCT}

\begin{corollary}[The classical conclusion of Lemma {\rm \ref{classical}}] \label{clcl}
Theorem {\rm \ref{bau-weak*}} independently implies the bounded sequentially
normal state formulation in Lemma {\rm \ref{classical}(3)}, and consequently the
full classical dominated convergence theorem in Lemma {\rm \ref{classical}(1)}. 
In fact, only the finite-trace a.u.\ case of Theorem {\rm \ref{bau-weak*}}
is needed for this classical application.
\end{corollary}
\begin{proof}
Let $(X,\Sigma,\mu)$ be a measure space, let $(h_n)$ be uniformly bounded
in $L^\infty(\mu)$ with $h_n\to h$ $\mu$-almost everywhere, and let
$\psi$ be a sequentially normal state on $L^\infty(\mu)$.
By Lemma~\ref{measure},
\[
 \nu(E)=\psi(\mathbf1_E),\qquad
 \psi(a)=\int_X a\,d\nu
\]
defines a countably additive probability measure $\nu\ll\mu$.
Egoroff's theorem gives almost uniform convergence for $\nu$.
Since multiplication is commutative, the integration state on
$L^\infty(\nu)$ is a faithful normal tracial state.
Theorem~\ref{bau-weak*}, applied to that algebra and that state, gives
weak* convergence. Using the same normal integration state,
we obtain
\[
 \psi(h_n)=\int h_n\,d\nu\longrightarrow\int h\,d\nu=\psi(h).
\]
This proves Lemma~\ref{classical}(3) without using the sequential weak* continuity theorem \ref{main}
 for general von Neumann algebras.

The implication $(3)\Rightarrow(1)$ in Lemma~\ref{classical}
then gives the full classical LDCT.
\end{proof}

\begin{remark}
On a finite measure space the  a.u.\ and b.a.u.\ notions we consider coincide
with almost-everywhere convergence. For a general localizable measure
they are  almost uniform convergence, which can be stronger than
almost-everywhere convergence. 
\end{remark}

\subsection{Local almost uniform convergence}
\label{ldct-addon:subsection}

We continue to fix a faithful normal semifinite weight $\varphi$ on $M$.
The following local a.u.\ conditions extend the preceding discussion. For a trace, these are the a.e.\ and b.a.e.\ conditions
of \cite[Definition~3.3]{BLS}. Their local spectral
definitions appear in \cite[Definition~4.11 and
Remark~4.12]{BLS}. Corollary~4.14 of that paper identifies
the one-sided local spectral condition with the a.e.\
condition of Definition~3.3.

\begin{definition}
\label{ldct-addon:definition}
For a sequence $(x_n)$ in $M$ and $x\in M$, we say that $x_n\to x$
\emph{locally $\varphi$-a.u.}, respectively \emph{locally
$\varphi$-b.a.u.}, if for every projection $p$ with
$\varphi(p)<\infty$ and every $\varepsilon>0$ there is a projection
$e\le p$ such that $\varphi(p-e)<\varepsilon$ and 
\[
 \|(x_n-x)e\|\longrightarrow0
 \qquad 
 ({\rm resp.}\ \|e(x_n-x)e\|\longrightarrow0).
\] 
The one-sided condition
implies the bilateral condition.
\end{definition}

The bilateral definition says exactly that $px_np\to pxp$ b.a.u.\ in
each finite-weight corner $pMp$.

For bounded sequences, Proposition~\ref{gm-order-local-au} identifies
the one-sided condition with $|x_n-x|\xrightarrow{o}0$; for
selfadjoint sequences it is equivalent to $x_n\xrightarrow{o}x$.
Consequently the one-sided implication to $\sigma$-strong and weak*
convergence below also follows from
\cite[\S2.2 and Proposition 2.7]{CHW}. The order implication does
not require $\sigma$-finiteness. We give a direct proof of the bilateral assertion 
 below.

\begin{theorem}[Local almost uniform convergence and integration]
\label{ldct-addon:theorem} 
Suppose $(x_n)$ is norm bounded, $x\in M$, and $y_n=x_n-x$.
\begin{enumerate}
\renewcommand{\theenumi}{\roman{enumi}}
\renewcommand{\labelenumi}{\textup{(\theenumi)}}
\item Local $\varphi$-b.a.u.\ convergence implies $x_n\to x$ weak* and
\begin{equation}
 \|h^{1/2}y_nh^{1/2}\|_1\longrightarrow0
 \quad\text{for every }h\in L^1(M)_+.
 \label{ldct-addon:l1}
\end{equation}
Here $L^1(M)$ is the Haagerup space. The same assertion holds for
$h\in L^1(M,\tau)_+$ when $\tau$ is a faithful normal semifinite trace.
The norm conclusion also holds under global $\varphi$-b.a.u.\ convergence.
\item Local $\varphi$-a.u.\ convergence implies $\sigma$-strong
convergence. In particular,
\begin{equation}
 y_n^*y_n\to0\ \text{weak*},\qquad |y_n|\to0\ \text{weak*}.
 \label{ldct-addon:deviations}
\end{equation}
If the adjoint sequence $(x_n^*)$ also converges locally $\varphi$-a.u., then
$x_n\to x$ $\sigma$-strong*. This holds in particular for selfadjoint
sequences and limits satisfying the local a.u.\ condition.
\end{enumerate}
\end{theorem}

\begin{proof}
Put $C=\sup_n\|y_n\|$. Fix a finite-weight projection $p$.
Choose the `local projections' 
$e_k\le p$ with $\varphi(p-e_k)<2^{-k}$. Lemma~\ref{small-weight-projections} gives
$e_k\to p$ strongly. Put $H=\operatorname{Tr}(h)$ and
$\omega_h(a)=\operatorname{Tr}(h^{1/2}ah^{1/2})$.  This is a positive normal functional with 
$\omega_h(1)=H$. 
Decomposing $y_n=ey_ne+(1-e)y_n+ey_n(1-e)$ and using H\"older gives
\begin{equation}
 \|h^{1/2}y_nh^{1/2}\|_1
 \le H\|ey_ne\|+2CH^{1/2}\omega_h(1-e)^{1/2}.
 \label{ldct-addon:estimate}
\end{equation}
Here $\|h^{1/2}(1-e)\|_2^2
=\|(1-e)h^{1/2}\|_2^2=\omega_h(1-e)$, and $\| h^{1/2} \|^2_2 = H$.
Take $e=e_k$, then let $n\to\infty$, $k\to\infty$, and finally
$p\to1$ strongly. This proves \eqref{ldct-addon:l1}; the tracial
argument is identical. Since
$|\omega_h(y_n)|\le\|h^{1/2}y_nh^{1/2}\|_1$ and normal positive
functionals span $M_*$, weak* convergence follows.
For global $\varphi$-b.a.u. convergence, use the same estimate with $e_k\to1$ strongly, as in Lemma~\ref{small-weight-projections}.

For (ii), Proposition~\ref{gm-order-local-au} gives order
convergence of $|y_n|$ and hence $\sigma$-strong convergence of $y_n$.
The error limits follow from Corollary~\ref{gm-modulus-state}.
Apply the same argument to $y_n^*$ for the strong* assertion.
\end{proof}

The next proposition combines the noncommutative Egoroff
subsequence theorem, in the form used in the proof of
\cite[Theorem~4.3(ii)]{CHW}, with a transfer of the cutoff estimates
to finite-weight corners. This relates convergence relative
to a faithful normal state to the local conditions for
$\varphi$.

In particular, the $\sigma$-finite strong-to-order subsequence
result here is already in \cite[Theorem 4.3(ii), with Proposition 2.1]{CHW}.
Apply it to $|x_n-x|$ and then Proposition~\ref{gm-order-local-au}
to obtain (iii). The proof below uses Egoroff for the aforementioned cutoff transfer.

\begin{proposition} 
\label{conv-cutoff-comparison}
Let $\varphi$ be a faithful normal semifinite weight on $M$, and
suppose that $(x_n)$ is norm bounded in $M$ and $x\in M$.
\begin{enumerate}[label=\textup{(\roman*)}]
\item Global $\varphi$-a.u.\ convergence implies local
$\varphi$-a.u.\ convergence. The corresponding implication also
holds for b.a.u.\ convergence.
\item If $M$ is $\sigma$-finite and $\rho$ is a faithful normal
state on $M$, then
\[
 \begin{aligned}
 x_n\to x\quad\rho\text{-a.u.}
 &\quad\Longleftrightarrow\quad
 x_n\to x\quad\text{locally }\varphi\text{-a.u.},\\
 x_n\to x\quad\rho\text{-b.a.u.}
 &\quad\Longleftrightarrow\quad
 x_n\to x\quad\text{locally }\varphi\text{-b.a.u.}
 \end{aligned}
\]
In particular, for a faithful normal state the global and local
versions of our a.u.\ and b.a.u.\ convergence coincide on norm-bounded sequences.
\item If $M$ is $\sigma$-finite and $x_n\to x$ $\sigma$-strongly,
then some subsequence converges to $x$ locally $\varphi$-a.u.
The subsequence may also be taken to converge $\rho$-a.u.\ 
for any fixed faithful normal state $\rho$ (this is already 
essentially in the older literature such as  {\rm \cite{HJP}}).
\end{enumerate}
If $x_n\to x$ $\sigma$-strong*, a common subsequence can be chosen
such that both it and its adjoint sequence converge locally
$\varphi$-a.u.\ to their respective limits.
\end{proposition}

For selfadjoint sequences, independence of b.a.u.\ convergence
in the algebra from the choice of faithful normal state was
already noted by Petz \cite[p.~250, following Theorem~2]{Petz}.
Part~\textup{(ii)} also compares this convergence with local
$\varphi$-b.a.u.\ convergence for a faithful normal semifinite
weight $\varphi$.

\begin{proof}
Put $y_n=x_n-x$ and $C=\sup_n\|y_n\|$.
We use the following form of noncommutative Egoroff's theorem:
if $a_j\to0$ $\sigma$-strongly, $p$ is a projection,
$\eta\in M_*^+$, and $\varepsilon>0$, then there exist a
projection $q\le p$ and a subsequence $(a_{j_k})$ such that
\[
 \eta(p-q)<\varepsilon,
 \qquad \|a_{j_k}q\|\longrightarrow0.
\]
See \cite[Theorem III.4.13, p.~85]{Tak1}, in the form stated in the
proof of \cite[Theorem 4.3]{CHW}.

First suppose that projections $e_j\to1$ strongly, and satisfy
$\|y_ne_j\|\to0$ for each fixed $j$.
Fix a projection $p$ with $\varphi(p)<\infty$ and a number
$\varepsilon>0$. The formula
$\eta(a)=\varphi(pap)$ for $a\in M_+$ defines a bounded normal
positive functional on $M$.
Apply Egoroff's theorem to $a_j=1-e_j$, the projection $p$,
and this functional. Since $(1-e_j)$ is norm bounded, its strong
convergence is also $\sigma$-strong convergence.
We obtain $q\le p$ with
\[
 \varphi(p-q)<\varepsilon,
 \qquad \|(1-e_{j_k})q\|\longrightarrow0.
\]
For each fixed $k$,
\[
 \|y_nq\|
 \le \|y_ne_{j_k}\|+C\|(1-e_{j_k})q\|.
\]
Taking the upper limit as $n\to\infty$, and then letting
$k\to\infty$, gives $\|y_nq\|\to0$.

A similar argument applies to bilateral cutoffs. Indeed, if instead
$\|e_jy_ne_j\|\to0$ for each $j$, then using $$q y_n q = q(1-e)y_n q +  qe y_n (1-e) q + qey_n eq,$$ we see that 
\[
 \|qy_nq\|
 \le \|e_{j_k}y_ne_{j_k}\|
       +2C\|(1-e_{j_k})q\|,
\]
and so $\|qy_nq\|\to0$ by the same successive limits.
We have therefore proved the local conclusions whenever such
strongly convergent good cutoffs $e_j$ exist.

For \textup{(i)}, choose global cutoffs $e_j$ with
$\varphi(1-e_j)<2^{-j}$. Lemma~\ref{small-weight-projections}
gives $e_j\to1$ strongly, so the preceding argument applies.

For the forward implications in \textup{(ii)}, choose global
$\rho$-cutoffs $e_j$ with $\rho(1-e_j)<2^{-j}$.
The same lemma, applied to the faithful normal state $\rho$,
gives $e_j\to1$ strongly, so the local $\varphi$-conclusions
again follow.
Conversely, suppose that the relevant local $\varphi$-condition
holds, and fix $\varepsilon>0$.
By Lemma~\ref{ldct-merged:finite-projections} and normality of
$\rho$, there is a projection $p$ such that
\[
 \varphi(p)<\infty,\qquad \rho(1-p)<\varepsilon/2.
\]
Choose local cutoffs $e_j\le p$ with
$\varphi(p-e_j)<2^{-j}$.
Lemma~\ref{small-weight-projections} gives $e_j\to p$ strongly,
and hence $\rho(p-e_j)\to0$.
For sufficiently large $j$,
\[
 \rho(1-e_j)=\rho(1-p)+\rho(p-e_j)<\varepsilon.
\]
This $e_j$ is the required global $\rho$-cutoff.
Taking $\varphi=\rho$ also proves the final assertion in
\textup{(ii)}.

For \textup{(iii)}, apply
\cite[Theorem 4.3(ii), with Proposition 2.1]{CHW} to $|y_n|$.
This gives an order-null subsequence. Proposition~\ref{gm-order-local-au}
gives local $\varphi$-a.u.\ convergence, and (ii) gives
$\rho$-a.u.\ convergence.

Finally, under $\sigma$-strong* convergence, apply
\textup{(iii)} first to $(x_n)$ and then to the adjoints along
the selected subsequence. Passing to a further subsequence
preserves the first local a.u.\ conclusion.
\end{proof}

\subsection{Local convergence in measure}

There are corresponding local convergence-in-measure conditions in which
the cutoff projection may vary with the index of the sequence. We continue
to fix a faithful normal semifinite weight $\varphi$ on $M$.

\begin{definition}[Local convergence in measure in the cutoff sense]
\label{lm-definition}
Let $(x_n)\subseteq M$, let $x\in M$, and put $y_n=x_n-x$.
We say that $x_n\to x$ \emph{locally in $\varphi$-measure in the
one-sided cutoff sense} if, for every projection $p\in M$ with
$\varphi(p)<\infty$, there are projections $e_n\le p$ such that \begin{equation*}
 \varphi(p-e_n)\longrightarrow0,
 \qquad \|y_ne_n\|\longrightarrow0.
 \tag{LM}\label{lm-one-sided}
\end{equation*}
The \emph{bilateral} version requires instead \begin{equation*}
 \varphi(p-e_n)\longrightarrow0,
 \qquad \|e_ny_ne_n\|\longrightarrow0.
 \tag{BLM}\label{lm-bilateral}
\end{equation*}
The \emph{symmetric} version of \eqref{lm-one-sided} means that both
$x_n\to x$ and $x_n^*\to x^*$ satisfy \eqref{lm-one-sided}.
The projections may depend on $p$ as well as on $n$.
\end{definition}

For a faithful normal semifinite trace $\tau$, conditions
\eqref{lm-one-sided} and \eqref{lm-bilateral} are precisely
sequential convergence, restricted to $M$, in the
$\tau$-locally measure topology and the weak
$\tau$-locally measure topology, respectively; see
\cite[Section~3]{CMlocal}. The definitions above use the
same cutoff conditions but for a general faithful normal
semifinite weight.

For $M=L^\infty(X,\mu)$ with $\varphi(f)=\int f\,d\mu$, all three
conditions coincide with local convergence in measure: for each measurable
$E$ of finite measure and each $\varepsilon>0$,
\[
 \mu\bigl(E\cap\{|x_n-x|>\varepsilon\}\bigr)\longrightarrow0.
\]
Indeed, the cutoff projections are characteristic functions of measurable
subsets of $E$.

\begin{proposition}[Bounded local measure convergence]
\label{lm-bounded-comparison}
Suppose that $(x_n)$ is norm bounded, $x\in M$, and $y_n=x_n-x$.
\begin{enumerate}[label=\textup{(\roman*)}]
\item Condition \eqref{lm-one-sided} is equivalent to
$\sigma$-strong convergence of $x_n$ to $x$.
\item The symmetric version of \eqref{lm-one-sided} is equivalent to
$\sigma$-strong* convergence. That is, condition \eqref{lm-one-sided} holds for both $x_n\to x$ and $x_n^*\to x^*$ if and only if $x_n\to x$ $\sigma$-strong*.  
It is also equivalent to requiring, for
each finite-weight projection $p$, projections $e_n\le p$ with \[
 \varphi(p-e_n)\longrightarrow0,
 \qquad
 \max\{\|y_ne_n\|,\|y_n^*e_n\|\}\longrightarrow0.
\]
\item Condition \eqref{lm-bilateral} implies
\begin{equation}
 \|h^{1/2}y_nh^{1/2}\|_1\longrightarrow 0
 \quad\text{for every }h\in L^1(M)_+,
 \label{lm-sandwiched}
\end{equation}
where $L^1(M)$ is the Haagerup space. In particular, it implies weak*
convergence. The same norm conclusion holds for every
$h\in L^1(M,\tau)_+$ if $\tau$ is a faithful normal semifinite trace.
\end{enumerate}
Thus, on norm-bounded sequences, the one-sided and symmetric conditions
are independent of the choice of $\varphi$. No $\sigma$-finiteness
assumption on $M$ is needed. Strong and strong* convergence in any faithful
normal concrete representation may replace $\sigma$-strong and
$\sigma$-strong* convergence, respectively.
\end{proposition}

\begin{proof}
Put $C=\sup_n\|y_n\|$. Suppose first that \eqref{lm-one-sided} holds,
and fix a finite-weight projection $p$. By
Lemma~\ref{small-weight-projections}, $p-e_n\to0$ strongly. Hence,
in a faithful normal representation and for each vector $\xi$,
\[
 \|y_np\xi\|
 \le \|y_ne_n\|\,\|\xi\|+C\|(p-e_n)\xi\|
 \longrightarrow0.
\]
Since
$\|y_n\xi\|\le\|y_np\xi\|+C\|(1-p)\xi\|$,
letting $p$ tend strongly to $1$ along the net in
Lemma~\ref{ldct-merged:finite-projections} gives strong convergence.
On bounded sets this is equivalent to $\sigma$-strong convergence.

Conversely, suppose that $x_n\to x$ $\sigma$-strongly. For a fixed
finite-weight projection $p$, the functional
$a\mapsto\varphi(pap)$ is bounded and normal, so
\[
 a_n=\varphi(py_n^*y_np)\longrightarrow0.
\]
Set $\varepsilon_n=(a_n+n^{-4})^{1/4}$ and take the spectral projection
\[
 e_n=\one_{[0,\varepsilon_n]}(|y_np|)
\]
with respect to  the corner $pMp$. 
Then $e_n\le p$, $\|y_ne_n\| = \|y_n p e_n\| \le\varepsilon_n$.  Since \[
\varepsilon_n^2\mathbf1_{(\varepsilon_n,\infty)}(t)\le t^2
\qquad t\ge0, 
\]
we obtain 
$\varepsilon_n^2(p-e_n)\le |y_n p|^2 =py_n^*y_np.$
Thus 
\[
 \varphi(p-e_n)
 \le \varepsilon_n^{-2}\varphi(py_n^*y_np)
 =\frac{a_n}{(a_n+n^{-4})^{1/2}}
 \le a_n^{1/2}\longrightarrow0.
\]
This proves (i).

Applying (i) to $y_n$ and $y_n^*$ proves the first assertion of (ii).
For common cutoffs under $\sigma$-strong* convergence, repeat the argument above but using instead
\[
 A_n=p(y_n^*y_n+y_ny_n^*)p,\qquad
 a_n=\varphi(A_n),\qquad
 e_n=\one_{[0,\varepsilon_n^2]}(A_n)
\]
in $pMp$, with $\varepsilon_n=(a_n+n^{-4})^{1/4}$.
Here $a_n\to0$, with the same estimate as before controlling $\varphi(p-e_n)$,
and both $\|y_ne_n\|$ and $\|y_n^*e_n\|$ are at most
$\varepsilon_n$. The converse follows directly from (i).

For (iii), fix $h\in L^1(M)_+$ and put
$H=\Tr(h)$ and $\omega_h(a)=\Tr(h^{1/2}ah^{1/2})$.
For a finite-weight projection $p$, take the cutoffs in
\eqref{lm-bilateral}. Then $e_n\to p$ strongly by
Lemma~\ref{small-weight-projections}. Applying
\eqref{ldct-addon:estimate} with $e=e_n$ gives
\[
 \limsup_n\|h^{1/2}y_nh^{1/2}\|_1
 \le 2CH^{1/2}\omega_h(1-p)^{1/2}.
\]
Let $p\to1$ strongly along the net in
Lemma~\ref{ldct-merged:finite-projections}. This proves
\eqref{lm-sandwiched}; the tracial proof is identical.
Finally, 
$|\omega_h(y_n)|\le\|h^{1/2}y_nh^{1/2}\|_1 \to 0$.  Since $\omega_h$ is the generic 
 positive normal functional, and the latter  span $M_*$, we have weak*
convergence.
\end{proof}

\begin{proposition}[Almost uniform convergence and moving cutoffs]
\label{lm-almost-uniform}
Local $\varphi$-a.u.\ convergence implies \eqref{lm-one-sided}, and
local $\varphi$-b.a.u.\ convergence implies \eqref{lm-bilateral}.
These implications do not require norm boundedness of the sequence. 
\end{proposition}

\begin{proof} 
To see the first implication, fix a projection $p$ with
$\varphi(p)<\infty$, and put $y_n=x_n-x$. For each $j\ge1$,
local almost uniform convergence supplies a projection
$e^{(j)}\le p$ such that
\[
 \varphi(p-e^{(j)})<2^{-j},
 \qquad
 \|y_ne^{(j)}\|\longrightarrow0 \; {\rm as} \;  n\to\infty.
\]
Choose positive integers $N_j$ successively, with
$N_{j+1}>N_j$, so that
$\|y_ne^{(j)}\|<2^{-j}$ whenever $n\ge N_j.$ 
Define
\[
 e_n=
 \begin{cases}
  0, & n<N_1,\\
  e^{(j)}, & N_j\le n<N_{j+1}.
 \end{cases}
\]
Then $e_n\le p$, and both $\varphi(p-e_n)$ and $\|y_ne_n\|$
tend to zero, proving \eqref{lm-one-sided}.

For the bilateral implication use the same ideas with
$\|e^{(j)}y_ne^{(j)}\|$ in place of $\|y_ne^{(j)}\|$.
\end{proof}

\begin{proposition}[Bilateral local measure convergence and subsequences]
\label{blm-local-subsequence}
Let $M$ be $\sigma$-finite, let $\varphi$ be a faithful normal
semifinite weight on $M$, and let $x_n,x\in M$ with
$\sup_n\|x_n\|<\infty$. The following are equivalent:
\begin{enumerate}[label=\textup{(\roman*)}]
\item $x_n\to x$ in $\mathrm{BLM}_{\varphi}$, that is,
condition \eqref{lm-bilateral} holds.
\item Every subsequence of $(x_n)$ has a further subsequence
converging locally $\varphi$-b.a.u.\ to $x$.
\end{enumerate}
For any fixed faithful normal state $\rho$ on $M$, the further
subsequence in \textup{(ii)} may also be required to converge
globally $\rho$-b.a.u.\ to $x$.
Consequently, on norm-bounded sequences in a $\sigma$-finite
von Neumann algebra, $\mathrm{BLM}_{\varphi}$ is independent of
the choice of the faithful normal semifinite weight $\varphi$.
\end{proposition}

\begin{proof}
Put $y_n=x_n-x$ and $C=\sup_n\|y_n\|$, and fix a faithful
normal state $\rho$ on $M$. Suppose first that \textup{(i)} holds.
Start with any subsequence and relabel it as $(x_n)$.
By Lemma~\ref{ldct-merged:finite-projections} and normality of
$\rho$, for each $j\ge1$ choose a projection $p_j$ such that
\[
 \varphi(p_j)<\infty,\qquad \rho(1-p_j)<2^{-j}.
\]
Condition \eqref{lm-bilateral} supplies projections
$r_n^{(j)}\le p_j$ with
\[
 \varphi(p_j-r_n^{(j)})\longrightarrow0,
 \qquad \|r_n^{(j)}y_nr_n^{(j)}\|\longrightarrow0.
\]
By Lemma~\ref{small-weight-projections},
$p_j-r_n^{(j)}\to0$ strongly, so
$\rho(p_j-r_n^{(j)})\to0$ for each fixed $j$.
Choose strictly increasing integers $N_j$ such that, whenever
$n\ge N_j$, both this last quantity and
$\|r_n^{(j)}y_nr_n^{(j)}\|$ are less than $2^{-j}$.
Set $e_n=r_n^{(j)}$ for $N_j\le n<N_{j+1}$, and set $e_n=1$
for $n<N_1$. We have $1-e_n = (1-p_j) +(p_j - r_n^{(j)})$.  Then
\[
 \rho(1-e_n)\longrightarrow0,
 \qquad \|e_ny_ne_n\|\longrightarrow0.
\]

Choose increasing indices $n_k$ such that, with $q_k=1-e_{n_k}$,
\[
 \rho(q_k)\le 2^{-k}k^{-2}.
\]
Thus $\sum_k \, k^2\rho(q_k)\le1$. Apply
Lemma~\ref{gm-resolvent-cutoff} with weight $\rho$ and
positive operators $q_k$. For every $\varepsilon>0$ this gives a
projection $f$ such that
\[
 \rho(1-f)<\varepsilon,
 \qquad \|q_kf\|\longrightarrow0.
\]
Since
$y_{n_k}-e_{n_k}y_{n_k}e_{n_k}
=q_ky_{n_k}+e_{n_k}y_{n_k}q_k$, after multiplying on both sides by $f$ we obtain
\[
 \|fy_{n_k}f\|
 \le \|e_{n_k}y_{n_k}e_{n_k}\|+2C\|q_kf\|
 \longrightarrow0.
\]
The indices $n_k$ were chosen before $\varepsilon$; only $f$
depends on $\varepsilon$. Hence this further subsequence
converges globally $\rho$-b.a.u.\ to $x$. By
Proposition~\ref{conv-cutoff-comparison}(ii), it also converges
locally $\varphi$-b.a.u.\ to $x$.

Conversely, for a fixed projection $p$ of finite $\varphi$-weight,
condition \eqref{lm-bilateral} is equivalent to the following:
for every $\delta>0$, all sufficiently large $n$ admit a
projection $e\le p$ with
\[
 \varphi(p-e)<\delta,\qquad \|ey_ne\|<\delta.
\]
Indeed, one direction is immediate. For the converse, choose
strictly increasing integers $N_j$ such that, for every
$n\geq N_j$, a projection satisfying both inequalities with
$\delta=2^{-j}$ exists. For $N_j\leq n<N_{j+1}$, choose such
a projection $e_n$. Then both $\varphi(p-e_n)$ and
$\|e_ny_ne_n\|$ tend to zero. 
If \textup{(i)} fails, there are therefore a fixed $p$, a
$\delta>0$, and a subsequence for which no such projection
exists at any index. By \textup{(ii)}, some further subsequence
converges locally $\varphi$-b.a.u.\ to $x$. A fixed projection
$e\le p$ witnessing that convergence with
$\varphi(p-e)<\delta$ has $\|ey_ne\|<\delta$ eventually along
the further subsequence, a contradiction.

Finally, Proposition~\ref{conv-cutoff-comparison}(ii) identifies
local $\varphi$-b.a.u.\ convergence with global
$\rho$-b.a.u.\ convergence on norm-bounded sequences, for every
faithful normal semifinite weight $\varphi$. Thus condition
\textup{(ii)}, and hence \textup{(i)}, is independent of
$\varphi$.
\end{proof}

In the tracial case, the equivalence between the bilateral
cutoff condition and $L^1$ convergence in every finite-trace
corner also follows from \cite[Propositions~6 and~9(1)]{CMlocal}.
The next proposition includes its equivalent formulation
using all positive integrable densities. 

\begin{proposition}[The bilateral measure condition in the tracial case]
\label{lm-tracial-bilateral}
Let $\tau$ be a faithful normal semifinite trace on $M$, let $(x_n)$
be norm bounded, and put $y_n=x_n-x$. The following are equivalent:
\begin{enumerate}[label=\textup{(\roman*)}]
\item Condition \eqref{lm-bilateral} holds with $\varphi=\tau$.
\item $\|h^{1/2}y_nh^{1/2}\|_1\to0$ for every
$h\in L^1(M,\tau)_+$.
\item $\|py_np\|_1\to0$ for every projection $p$ with
$\tau(p)<\infty$.
\end{enumerate}
\end{proposition}
\begin{proof}
Proposition~\ref{lm-bounded-comparison}(iii) proves (i)$\Rightarrow$(ii),
and taking $h=p$ proves (ii)$\Rightarrow$(iii).
For (iii)$\Rightarrow$(i), fix such a $p$, put $b_n=py_np$, and set
$\delta_n=(\|b_n\|_1+n^{-2})^{1/2}$. In $pMp$ take
$e_n=\one_{[0,\delta_n]}(|b_n|)$. Then $e_n |b_n|^2 e_n \leq \delta_n^2 \, e_n$, so that
$\|b_ne_n\|\le\delta_n$.  
\[
 \tau(p-e_n)\le\delta_n^{-1}\|b_n\|_1\longrightarrow0,
 \qquad
 \|e_ny_ne_n\|=\|e_nb_ne_n\|
 \le\|b_ne_n\|\le\delta_n\longrightarrow0,
\]
by definition of $\delta_n$ and since $\|py_np\|_1\to0$. 
\end{proof}

\begin{remark}[Relation to existing measure topologies]
\label{lm-terminology}
For the usual trace on $B(H)$, the one-sided cutoff condition
\eqref{lm-one-sided} is equivalent to strong operator
convergence, while the bilateral condition
\eqref{lm-bilateral} is equivalent to weak operator
convergence. Thus the one-sided and bilateral conditions
can differ even on norm-bounded sequences. 

If $\tau(1)<\infty$, then, on norm-bounded sequences, the one-sided,
bilateral, and symmetric conditions all coincide with
$\|x_n-x\|_1\to0$ and with $\sigma$-strong* convergence.
Indeed, taking $p=1$ in Proposition~\ref{lm-tracial-bilateral}
gives the equivalence of the bilateral condition and the $L^1$ criterion, and
$\tau(y_n^*y_n)=\tau(y_ny_n^*)\le C\tau(|y_n|) \to 0$,
where $C=\sup_n\|y_n\|$. So we have ($\sigma$-)strong* convergence, which is equivalent to the one-sided condition. 
Indeed, bounded strong convergence
implies the one-sided condition and hence implies the bilateral condition.
In the  faithful normal state case, 
suppose that $M$ is $\sigma$-finite and let $\rho$ be a faithful normal
state. For a norm-bounded sequence $(x_n)$ and $x\in M$, the one-sided
condition \eqref{lm-one-sided}, for any faithful normal semifinite
weight $\varphi$, is equivalent to
\[
 \rho\bigl(\one_{(\varepsilon,\infty)}(|x_n-x|)\bigr)
 \longrightarrow0
 \qquad\text{for every }\varepsilon>0.
\]
Indeed, by Lemma~\ref{small-weight-projections} with weight $\rho$,
this condition forces the spectral projections to converge strongly,
and hence weak*, to zero. The equivalence follows from
Lemma~\ref{prop:projection-measure-convergence} and
Proposition~\ref{lm-bounded-comparison}. Requiring the same condition
for $x_n^*-x^*$ gives the symmetric version and hence
$\sigma$-strong* convergence. 
\end{remark}

We now give some applications to dominated convergence.

\paragraph{Local dominated convergence.}
Suppose that $(x_n)$ is norm bounded in $M$ and that
$x_n\to x\in M$ locally $\varphi$-a.u. Put $y_n=x_n-x$.
By Corollary~\ref{gm-modulus-state}, every sequentially
normal state $\psi$ on $M$ satisfies
\[
  \psi(|y_n|)\longrightarrow0,
  \qquad
  \psi(y_n^*y_n)\longrightarrow0.
\]
By scaling, the same conclusions hold for every bounded
sequentially normal positive functional. In particular,
if $g$ is a positive affiliated operator with
$\varphi(g)<\infty$, then
\[
  (g^{1/2}\varphi g^{1/2})(|y_n|)\longrightarrow0,
  \qquad
  (g^{1/2}\varphi g^{1/2})(y_n^*y_n)\longrightarrow0,
\]
since $g^{1/2}\varphi g^{1/2}$ is a bounded normal positive
functional by Lemma~\ref{ldct-merged:integral-construction}.

Under local $\varphi$-b.a.u.\ convergence we also obtain
the bounded-domination conclusion of
Proposition~\ref{bounded-weight}.
Explicitly, let $\Phi$ be a sequentially normal weight,
let $g\in M_+$ satisfy $\Phi(g)<\infty$, and suppose that
$f_n=f_n^*\in M$ satisfy
\[
  -g\leq f_n\leq g,
  \qquad
  f_n\longrightarrow f\in M
  \quad\hbox{locally $\varphi$-b.a.u.}
\]
Then
\[
  -g\leq f\leq g,\qquad
  f_n,f\in\mathfrak m_\Phi,\qquad
  \Phi(f_n)\longrightarrow\Phi(f).
\]
Indeed, $\|f_n\|\leq\|g\|$, so
Theorem~\ref{ldct-addon:theorem}(i) gives weak*
convergence and Proposition~\ref{bounded-weight} applies.

\begin{corollary}[Varying majorants]
\label{ldct-addon:varying}
Suppose $\Phi$ is sequentially weak* lower semicontinuous on $M_+$,
as is every normal weight. Let $a_n,a\in M_+$ and $f_n=f_n^*\in M$
satisfy
\[
 \sup_n\|a_n\|<\infty,\qquad -a_n\le f_n\le a_n,\qquad
 \Phi(a_n)<\infty,\qquad \Phi(a_n)\to\Phi(a)<\infty.
\]
If $a_n\to a$ and $f_n\to f\in M$ locally $\varphi$-b.a.u., then
$-a\le f\le a$, all $f_n,f\in\mathfrak m_\Phi$, and
$\Phi(f_n)\to\Phi(f)$.
Moreover, suppose that $a_n\to a$ locally $\varphi$-b.a.u.,
and let $(z_n)$ be a norm-bounded sequence in $M$ such that
$z_n\to z\in M$ locally $\varphi$-a.u. If $|z_n-z|\le a_n$
for every $n$, then 
$\Phi(|z_n-z|)\longrightarrow0.$
\end{corollary}

\begin{proof} Since $\|f_n\|\le\|a_n\|$, both sequences are norm
bounded. Theorem~\ref{ldct-addon:theorem}(i) gives
$a_n\to a$ and $f_n\to f$ weak*. 
Apply Theorem~\ref{thm:varying}.
For the last assertion, Theorem~\ref{ldct-addon:theorem}(ii) gives
$z_n\to z$ $\sigma$-strongly, so Corollary~\ref{cor:absolute} applies.
\end{proof}

\begin{remark} 
\label{ldct-addon:distinction}
In the commutative localizable case the local a.u.\ and b.a.u.\
conditions coincide and are equivalent to a.e.\ convergence: apply
Egoroff on each finite-measure set, and use semifiniteness for the
converse. Hence bounded a.e.\ convergence satisfies both parts of
Theorem~\ref{ldct-addon:theorem}, without a global Egoroff assumption. 
 The normalization argument in the proof of Corollary \ref{clcl}  then yields the full classical dominated convergence theorem. 
\end{remark}

\subsection*{Comparison with spectral convergence}
\label{sec:spectral-comparison} 

We conclude this section by comparing the global cutoff
conditions with spectral formulations of convergence. 
For a trace, our
definitions presented in Definition \ref{cutoff} are exactly the ordinary a.u.\ and b.a.u.\ notions recalled
in \cite[Definition 3.1]{BLS}. These should be distinguished from the
\emph{spectral} notion called $\tau$-a.u.\ in \cite[Definition 4.2]{BLS}.
Proposition 4.10 there shows that spectral $\tau$-a.u.\ convergence implies
ordinary a.u.\ convergence, with a subsequential converse.  
Proposition 4.13 there identifies approximately $\tau$-a.u.\ convergence with
ordinary a.u.\ convergence. These comparisons do not require a uniform
operator-norm bound. The finite-weight projection observation underlying
Lemma~\ref{small-weight-projections} also occurs in the proof of
\cite[Proposition 5.3, p.~19]{BLS}, where it is obtained using Haagerup
$L^1$-duality. Their paper also studies convergence of unbounded operators
and in Haagerup $L^p$-spaces. Our emphasis here is on the
consequences for integration against sequentially normal
states and weights, and on dominated convergence in the
predual. 

For the bounded convergence conclusion needed here, the relevant spectral
formulation has the following direct nontracial variant. 

\begin{proposition}[A bounded nontracial spectral variant]
\label{bls-spectral-variant}
Let $\varphi$ be a faithful normal semifinite weight on $M$, and suppose
$x_n,x\in M$ with $\sup_n\|x_n\|<\infty$. Assume that for every
$\varepsilon,\delta>0$ there exists a projection $e\in M$ such that
\[
 \varphi(1-e)<\delta,\qquad
 \varphi\!\left(\bigvee_{m\ge N}
  1_{(\varepsilon,\infty)}\bigl(|(x_m-x)e|\bigr)\right)
       \longrightarrow0\quad(N\to\infty).
\]
Then $x_n\to x$ $\sigma$-strongly, hence weak*, and
$\psi(x_n)\to\psi(x)$ for every sequentially normal state $\psi$.
The hypothesis holds in particular under our $\varphi$-a.u.\
convergence, or under the spectral condition
\[
 \varphi\!\left(\bigvee_{m\ge N}
  1_{(\varepsilon,\infty)}(|x_m-x|)\right)\longrightarrow0
 \quad\text{for every }\varepsilon>0.
\]
\end{proposition}
\begin{proof}
Put $y_n=x_n-x$ and $C=\sup_n\|y_n\|$ and pass to a faithful normal
standard representation of $M$. Fix $\varepsilon,\delta>0$, and let $e=e_{\varepsilon,\delta}$
be the projection given by the hypothesis, so that
$\varphi(1-e)<\delta$. Now set
\[
 q_N=\bigvee_{m\ge N}1_{(\varepsilon,\infty)}(|y_me|).
\]
Then $q_N\le e$, since $|y_me|$ is supported on $e$, and
$q_N\to0$ strongly by Lemma~\ref{small-weight-projections}.
For $n\ge N$, spectral calculus gives
$\|y_ne(1-q_N)\|\le\varepsilon$. Hence
\[
 \|y_n\xi\|\le\varepsilon\|\xi\|
       +C\|q_N\xi\|+C\|(1-e)\xi\|.
\]
Taking the upper limit in $n$ and then letting $N\to\infty$ leaves
the first and third terms. For a fixed $\varepsilon$ and $(\delta_n)$ such that $\delta_n\to 0$, 
Lemma~\ref{small-weight-projections} then ensures that $\|(1-e_{\delta_n})\xi\|\to 0$. Finally let $\varepsilon\downarrow0$.
This proves strong convergence in the standard representation, equivalently
$\sigma$-strong convergence on this bounded sequence. Weak* convergence
follows, and Theorem~\ref{main} gives the assertion about sequentially
normal states.

Under our a.u.\ convergence, choose $e$ with $\|y_ne\|\to0$;
the displayed tail spectral projections are eventually zero. Under the
last spectral condition, take $e=1$.
\end{proof}

\begin{remark} 
The hypothesis of Proposition~\ref{bls-spectral-variant} is the direct
weight analogue of \cite[Definition 4.11]{BLS}, restricted to bounded
sequences. \end{remark}

\section{Conclusions} Our objective has been to formulate various convergence theorems, such as versions of Lebesgue's dominated convergence theorem, for quantum integration.
We work usually on completely general von Neumann algebras, however our 
 results recover, and thus generalize, the classical theorems.  We also developed parts of the accompanying
theory. 
We have treated
integrable majorants in the tracial and Haagerup $L^1$ settings
and integrals associated with normal semifinite weights, and
established implications and equivalences among the various
modes of convergence considered here.  The latter are similar to the classical case summarized at the end of the Introduction.
 In the noncommutative case we distinguish (as is also done e.g.\ in \cite{BLS}) between the  one-sided and the bilateral cases of a.u.\ and local measure convergence.  Our global a.u.\ variants imply the local ones.
The dominated convergence theorem that we establish in the predual of $M$ gives norm convergence under
fixed positive order domination and almost uniform convergence,
without a trace or $\sigma$-finiteness assumption on the ambient algebra.
 For a fixed Haagerup $L^1$ majorant, the
normalization and cutoff-transfer results also connect this formulation
with bounded convergence in a $\sigma$-finite supporting corner.

For clarity we include some of  the order comparisons in the diagram below.
Here sequences are norm bounded.   To make the chart fit into a small space
we have used some random symbols for various convergences,
 and we hope the reader will forgive this:
$A_\varphi$ and $B_\varphi$ denote
local $\varphi$-a.u.\ and local $\varphi$-b.a.u.\ convergence, and
$O_{|\cdot|}$ denotes $|x_n-x|\xrightarrow{o}0$. Write LM and BLM
for \eqref{lm-one-sided} and \eqref{lm-bilateral}, and $Q$ for
the sandwiched $L^1$ conclusion \eqref{ldct-addon:l1}. Then
\[
\begin{array}{ccccccc}
 \text{global a.u.}&\Longrightarrow&A_\varphi&\Longleftrightarrow&O_{|\cdot|}
 &\Longrightarrow&\sigma\text{-strong}\ \Longleftrightarrow\ \mathrm{LM}\\[4pt]
 \Downarrow&&\Downarrow&&&&\Downarrow\\[4pt]
 \text{global b.a.u.}&\Longrightarrow&B_\varphi&\Longrightarrow&\mathrm{BLM}
 &\Longrightarrow&Q\ \Longrightarrow\ \text{weak*}.
\end{array}
\]
These implications hold for arbitrary $M$. For selfadjoint sequences
$O_{|\cdot|}$ is equivalent to $x_n\xrightarrow{o}x$.
In the $\sigma$-finite case, every strongly convergent bounded sequence
has, from each subsequence, a further locally a.u.-convergent subsequence,
by the cited  theorem from \cite{CHW} and the local/order comparison.
Proposition~\ref{conv-cutoff-comparison}(ii) identifies local convergence
with global convergence for a faithful normal state.
The bilateral conclusion $Q$ does not imply
$\rho(|x_n-x|)\to0$, even for a faithful normal state
$\rho$. Indeed, let $y_n=e_{n1}+e_{1n}$ in $B(\ell^2)$
and let
$\rho(a)=\sum_{j\ge1}2^{-j}\langle a\xi_j,\xi_j\rangle$.
The projections $p_N=\sum_{j=1}^N e_{jj}$ satisfy
$\rho(1-p_N)=2^{-N}$ and $p_Ny_np_N=0$ for $n>N$.
Thus $y_n\to0$ globally $\rho$-b.a.u., and hence satisfies
$Q$, whereas
\[
  |y_n|=e_{11}+e_{nn},\qquad
  \rho(|y_n|)=\tfrac12+2^{-n}\qquad n\ge2.
\]

The convergence results in our paper  may be useful in quantum measure theory wherever one must justify limits of
expectations, or for approximating integrable observables when
no trace is available on the original algebra.

In a sequel we plan to develop the corresponding global  convergence in measure theory.
We briefly indicate its main results. Let $\varphi$ be a faithful
normal semifinite weight on an arbitrary von Neumann algebra $M$,
and let $x_n,x\in M$ with $\sup_n\|x_n\|<\infty$.
Put $y_n=x_n-x$. Consider the conditions
\[
\begin{aligned}
\mathrm{G}_{\varphi}:\quad&
 \varphi(1-e_n)\longrightarrow0,\qquad
 \|y_ne_n\|\longrightarrow0,\\
\mathrm{BG}_{\varphi}:\quad&
 \varphi(1-e_n)\longrightarrow0,\qquad
 \|e_ny_ne_n\|\longrightarrow0,
\end{aligned}
\]
where in each case $(e_n)$ is a suitable sequence of projections.
These are global convergence-in-measure conditions in the
one-sided and bilateral cutoff senses.

The condition $\mathrm{G}_{\varphi}$ is equivalent to every
subsequence having a further subsequence converging globally
$\varphi$-a.u.\ to $x$. The analogous characterization of
$\mathrm{BG}_{\varphi}$ uses global $\varphi$-b.a.u.\ convergence. 
Moreover, $\mathrm{G}_{\varphi}$ implies $\sigma$-strong
convergence, while $\mathrm{BG}_{\varphi}$ implies weak*
convergence and
\[
 \|h^{1/2}(x_n-x)h^{1/2}\|_1\longrightarrow0
 \qquad h\in L^1(M)_+,
\]
where $L^1(M)$ is the Haagerup space.

If $M$ is $\sigma$-finite and $\rho$ is any faithful normal
state, then, for the bounded sequences under consideration,
\[
 \mathrm{LM}_{\varphi}\ \Longleftrightarrow\ \mathrm{G}_{\rho},
 \qquad
 \mathrm{BLM}_{\varphi}\ \Longleftrightarrow\ \mathrm{BG}_{\rho}.
\] 
For a faithful normal semifinite trace $\tau$, the conditions
$\mathrm{G}_{\tau}$ and $\mathrm{BG}_{\tau}$ above both coincide with
ordinary global convergence in $\tau$-measure.
 The tracial a.u.\ subsequence extraction
result is due to Goldstein--Litvinov
\cite[Proposition~1]{GL}; see also the proof of
\cite[Corollary~3.8]{BLS}. The sequel will give the nontracial
proofs and further comparisons with the local convergence
notions considered here.

Natural next steps are to relate these sequential criteria more
systematically to existing theories of noncommutative uniform
integrability and convergence in measure, and to study convergence
when the integrating weights or states also vary. For unbounded
observables, it would be useful to clarify the relation between
convergence in affiliated-operator models and convergence in
Haagerup $L^p$ spaces. Such developments could turn these
structural continuity results into tools for approximation and
stability in noncommutative integration. As indicated in
\cite{BLd}, we also hope to investigate extensions of some of
these results to monotone $\sigma$-complete $C^*$-algebras,
and to explore applications to concrete problems in
noncommutative ergodic theory and other settings in which
sequential convergence plays a central role.

\bigskip

{\bf Acknowledgements:}  We thank Louis Labuschagne and Claud Steyn for several discussions related to various notions of sequential convergence (usually in a tracial situation), and we thank Louis for teaching us many of the basics of the subject over several years, proposing some variants of results here, 
for giving many useful comments and suggestions on various drafts of this paper.    We
were very influenced throughout by \cite{BLS}, and its perspectives and objectives particularly concerning 
nc versions of almost uniform convergence and tracial versions of convergence in measure.   We thank Brent Nelson for a literature reference on semifinite weights.  After the first arXiv version we had a very helpful conversation with Marten Wortel, during which I realized that there was substantially more overlap with the results of  \cite{CHW} than I had previously thought.   This led to the Appendix and an associated revision and shortening of parts of the paper (also, some additional results were added).
 D.P.B acknowledges support from NSF Grant DMS-2154903, and also from CoE-MaSS Ref No. 2025-048-OPA-Quantum Measures, and the School of Mathematical and Statistical Sciences at NWU Potchefstroom, which helped make some of the discussions mentioned above possible.  Although we had established versions of several results here over several years, most of the present formulations of results here were developed recently through sustained, highly interactive and iterative mathematical discussions between the human author and ChatGPT, initially using Sol 5.6 and subsequently 6.0 Astra. 
 The human author directed and decided the mathematics, suggested most of the theorems and directions and usually provided references and suggestions, and double checked and corrected results.   
 ChatGPT provided substantive assistance with the mathematical development, proposed alternative formulations and variants, constructed and adapted numerous  proofs and arguments, and helped identify and correct errors in arguments supplied by the author. It also assisted with literature searches, exposition, typesetting, and restructuring the organization of the paper  to its final form.  The human author assumes full responsibility for the correctness of all mathematical statements, and for the integrity and accuracy of all citations.

\appendix
\section{Order convergence and local almost uniform convergence}
\label{app:order-local}

We prove that, for norm-bounded sequences, local
$\varphi$-a.u.\ convergence of $x_n$ to $x$ is equivalent
to sequential order convergence of $|x_n-x|$ to zero.
For selfadjoint sequences and limits, these conditions
are also equivalent to sequential order convergence
of $x_n$ to $x$.

 For $x_n,x\in M_{\sa}$,
sequential order convergence $x_n\xrightarrow{o}x$ means that there
are $h_n\in M_+$ with
\[
 h_n\downarrow0,\qquad -h_n\le x_n-x\le h_n.
\]
This is equivalent to the usual increasing lower and decreasing upper
bounds in $M_{\sa}$; see \cite[\S2.1]{CHW}.
It implies $\sigma$-strong convergence without $\sigma$-finiteness
\cite[\S2.2, p.~103, and Proposition 2.7]{CHW}.
For completeness, the useful estimate is
\begin{equation}\label{gm-order-square}
 -h\le y\le h,\quad y=y^*,\ h\ge0
 \quad\Longrightarrow\quad y^2\le\|h\|h.
\end{equation}
Indeed $(h+\epsilon1)^{-1/2}y(h+\epsilon1)^{-1/2}$ is a selfadjoint
contraction; multiplication and then $\epsilon\downarrow0$ give the
estimate. Thus $y_n^2\le\|h_1\|h_n$ for an order-null sequence,
which proves strong convergence by applying normal positive functionals.
Operator monotonicity of square root also gives
$|y_n|\le\|h_1\|^{1/2}h_n^{1/2}\downarrow0$. Consequently selfadjoint
order convergence is equivalent to order convergence of the modulus
of the error.

\begin{lemma}[A common cutoff for weighted positive operators]
\label{gm-resolvent-cutoff}
Let $\psi$ be a normal weight on a von Neumann algebra $N$, and let
$a_k\in N_+$ satisfy
\[
 \sum_{k=1}^{\infty} k^2\psi(a_k)\le1.
\]
For each $\varepsilon>0$ there is a projection $f\in N$ such that
\[
 \psi(1-f)<\varepsilon,
 \qquad \|a_k^{1/2}f\|\longrightarrow0.
\]
\end{lemma}

\begin{proof}
Fix $t>0$, set $A_m=\sum_{k=1}^m k^2a_k$, and put
$B_t=\operatorname*{s-lim}_m(1+tA_m)^{-1}$.
This is a decreasing strong limit, since $(A_m)$ is increasing, and we have $0\le B_t\le1$.
Since $1-(1+tA_m)^{-1}\le tA_m$, normality gives
\begin{equation}\label{gm-resolvent-weight}
 \psi(1-B_t)=\sup_m\psi\bigl(1-(1+tA_m)^{-1}\bigr)\le t.
\end{equation}
Let $f=\mathbf1_{[1/2,1]}(B_t)$. Then $1-f\le2(1-B_t)$ and
$f\le2B_t$, so $\psi(1-f)\le2t$.
For every $k$ and $m\ge k$, we have $tk^2a_k\le tA_m$,
and so  $(1+tA_m)^{-1}\le(1+tk^2a_k)^{-1}.$ 
Passing to the strong limit gives
$B_t\le(1+tk^2a_k)^{-1}$.
 Hence
\[
 \|a_k^{1/2}f\|^2
 =\|a_k^{1/2}fa_k^{1/2}\|
 \le2\|a_k^{1/2}B_ta_k^{1/2}\|
 \le\frac{2}{tk^2}.
\]
Choose $0<t<\varepsilon/2$. This proves both assertions.
\end{proof}

\begin{proposition}[Bounded local a.u.\ and order]\label{gm-order-local-au}
Let $M$ be a  von Neumann algebra, $\varphi$ a faithful normal semifinite weight on $M$, and $x_n,x\in M$ with $\sup_n\|x_n\|<\infty$. Then
\begin{equation}\label{gm-modulus-bridge}
 x_n\to x\text{ locally }\varphi\text{-a.u.}
 \quad\Longleftrightarrow\quad |x_n-x|\xrightarrow{o}0.
\end{equation}
We recall that local convergence means that for every projection $p$ with
$\varphi(p)<\infty$ and every $\epsilon>0$ there is a projection
$e\le p$ such that $\varphi(p-e)<\epsilon$ and
$\|(x_n-x)e\|\to0$. If $x_n,x$ are selfadjoint, these are equivalent to $x_n\xrightarrow{o} x$. Consequently this local a.u.\ condition is independent of the chosen faithful normal semifinite weight.
\end{proposition}
\begin{proof}
Put $y_n=x_n-x$, $z_n=|y_n|$, and $C=\sup_n\|z_n\|$. For every projection $e$,
\[
 \|y_ne\|^2=\|ey_n^*y_ne\|=\|ez_n^2e\|=\|z_ne\|^2.
\]
It therefore suffices to treat the positive sequence $z_n$.

Suppose $0\le z_n\le h_n\downarrow0$. Fix a finite-weight projection $p$. The map $a\mapsto\varphi(pap)$ is bounded and normal, so $\varphi(ph_np)\to0$. Select $n_k\uparrow\infty$ with $\sum k^2\varphi(ph_{n_k}p)\le1$. Apply Lemma~\ref{gm-resolvent-cutoff} inside $pMp$ to $ph_{n_k}p$. This gives $e\le p$, $\varphi(p-e)<\epsilon$ and $\|eh_{n_k}e\|\to0$. By monotonicity $\| e h_n e \| \to 0$. Since $z_n^2\le Cz_n\le Ch_n$, we get $\|z_ne\|\to0$.

The converse is implicit in \cite{CHW} in the $\sigma$-finite case.
To do the general case, let $\mathcal E$ be the set of projections $e$ with $\|z_ne\|\to0$. Local convergence and Lemma~\ref{small-weight-projections} imply $\bigvee\mathcal E=1$.
Indeed, fix a projection
$p$ with $\varphi(p)<\infty$, and put
\[
q=\bigvee\{e\in\mathcal E:e\le p\}.
\]
By local $\varphi$-a.u.\ convergence, for every $\varepsilon>0$
there is an $e\in\mathcal E$ with $e\le p$ and
$\varphi(p-e)<\varepsilon$. Since $e\le q\le p$, we have
\[
0\le\varphi(p-q)\le\varphi(p-e)<\varepsilon.
\]
Thus $\varphi(p-q)=0$, and faithfulness gives $q=p$.
Hence every finite-weight projection is a supremum of
projections in $\mathcal E$.  By semifiniteness, the finite-weight projections themselves have supremum $1$.
 Consequently $\bigvee\mathcal E=1$.

We construct orthogonal $\sigma$-finite projections reducing
the sequence $(z_n)$. Put
$u_n=\exp(iz_n/(C+1))$, and let $\Gamma$ be the group of finite
words in the $u_n$ and their inverses. This group is countable.
Since $\sigma(z_n/(C+1))\subset[0,1)$, continuous functional
calculus gives $z_n=(C+1)\operatorname{Arg}(u_n)$, using the
continuous argument function on the spectral arc. Thus a
projection commuting with every $u_n$ also commutes with every
$z_n$.

Let $r$ be a nonzero projection commuting with $\Gamma$, and
choose a normal state $\rho$ with support $p_0\le r$. Set
\[
 q=\bigvee_{u\in\Gamma} \, u \, p_0 \, u^*\le r.
\]
Choose $c_u>0$ with $\sum_{u\in\Gamma}c_u=1$. The normal state
$\eta=\sum_{u\in\Gamma}c_u\rho(u^*\,\cdot\,u)$ has support $q$:
a projection is $\eta$-null exactly when it is null for every
summand. Hence $qMq$ admits a faithful normal state and is
$\sigma$-finite. Moreover, for $v\in\Gamma$,
\[
 vqv^*=\bigvee_{u\in\Gamma}(vu)p_0(vu)^*=q,
\]
by reindexing $\Gamma$. Thus $q$ commutes with every $z_n$.
Take a maximal orthogonal family $(q_\alpha)$ of such
$\sigma$-finite projections. Their sum is $1$, since otherwise
the nonzero complementary projection commutes with $\Gamma$
and the same construction gives another member, a contradiction.
(C.f.\ e.g.\ the proof of \cite[Proposition 3.6]{BHK}.)

Fix $q=q_\alpha$. The supports of $qeq$, $e\in\mathcal E$,
have join $q$. Let $\omega$ be a faithful normal state on
$qMq$. For each finite subset $F\subset\mathcal E$, put
\[
 r_F=\bigvee_{e\in F}s(qeq).
\]
Then $r_F\uparrow q$, so normality allows us to choose finite
sets $F_k\subset\mathcal E$ with
$\omega(q-r_{F_k})<2^{-k}$. Enumerate their union as $(e_j)$,
repeating terms if necessary. If
$r=\bigvee_j s(qe_jq)$, then
\[
 0\le\omega(q-r)\le\omega(q-r_{F_k})<2^{-k}
 \qquad\text{for every }k.
\]
Faithfulness therefore gives $r=q$. Set
\[
 a=\sum_{j\ge1}2^{-j}qe_jq,
 \qquad
 p_m=\mathbf{1}_{[1/m,\infty)}(a).
\]
Then $s(a)=q$ and $p_m\uparrow q$. Moreover, since $q$
commutes with every $z_n$ and $e_j\in\mathcal E$,
\[
 \|z_nqe_jq\|
   =\|qz_ne_jq\|
   \le \|z_ne_j\|\longrightarrow0
 \qquad\text{for each fixed }j.
\]
The tails are uniformly small: writing
$C=\sup_n\|z_n\|$, we have
\[
 \left\|z_n\sum_{j>J}2^{-j}qe_jq\right\|
 \le C\sum_{j>J}2^{-j}
 =C2^{-J}.
\]
Thus
\[
 \|z_na\|
 \le \sum_{j=1}^{J}2^{-j}\|z_nqe_jq\|+C2^{-J}.
\]
First letting $n\to\infty$ and then $J\to\infty$ gives
$\|z_na\|\to0$. Hence
$\|z_np_m\|\le m\|z_na\|\to0$ for each fixed $m$.

Choose strictly increasing $N_m$ such that $\|z_np_m\|\le1/m$ for $n\ge N_m$. In $qMq$, write $r_m=q-p_m$ and $z_n^{(q)}=qz_nq$. The decomposition
\[
 z_n^{(q)}=r_mz_n^{(q)}r_m+z_n^{(q)}p_m+p_mz_n^{(q)}r_m
\]
gives $0\le z_n^{(q)}\le Cr_m+(2/m)q$ when $n\ge N_m$. If $N_m\le n<N_{m+1}$ put $h_n^{(q)}=Cr_m+(2/m)q$, and use $(C+2)q$ to bound the initial terms. Then
\[
 0\le z_n^{(q)}\le h_n^{(q)}\downarrow0,
 \qquad h_n^{(q)}\le(C+2)q.
\]
The strong sums $h_n=\sum_\alpha h_n^{(q_\alpha)}$ exist, decrease to zero, and dominate $z_n$, since the $q_\alpha$ commute with  all $z_n$, and $\sum_\alpha \, z_n^{(q_\alpha)} = z_n$. This proves $z_n\xrightarrow{o}0$.

The selfadjoint assertion follows from \eqref{gm-order-square}
and the modulus comparison preceding Lemma~\ref{gm-resolvent-cutoff}.
\end{proof}

The compression-to-envelope estimate in this proof is done 
in \cite[Theorem 4.3(ii), p.~111]{CHW}. We added the reduction from our local
weight condition to the reducing corners. 

\begin{corollary} \label{gm-modulus-state}
Under the hypotheses of Proposition {\rm \ref{gm-order-local-au}}, put
$y_n=x_n-x$. If $x_n\to x$ locally $\varphi$-a.u., then every
sequentially normal state $\psi$ satisfies
\[
 \psi(|y_n|)\to0,\qquad \psi(y_n^*y_n)\to0,
 \qquad \psi(y_n)\to0.
\]
\end{corollary}
\begin{proof}
Take $0\le|y_n|\le h_n\downarrow0$ and $C=\sup_n\|y_n\|$.
Then $\psi(|y_n|)\le\psi(h_n)\to0$,
$\psi(y_n^*y_n)\le C\psi(h_n)\to0$, and
$|\psi(y_n)|^2\le\psi(y_n^*y_n)$. 
\end{proof}


\begin{thebibliography}{99}

\bibitem{ABB}
C. A. Akemann and G. A. Bagheri--Bardi,
\emph{A stronger noncommutative Egoroff's theorem},
Nihonkai Math. J. \textbf{25} (2014), 65--68.

\bibitem{ACW}

 
K.~Abela, E.~Chetcuti, and H.~Weber,
\emph{On different modes of order convergence and some applications},
Positivity \textbf{26} (2022), article~14, doi:10.1007/s11117-022-00885-2, 
arXiv:2012.13752v1.
 

\bibitem{Arveson}
W.~B.~Arveson, \emph{Subalgebras of $C^*$-algebras},
Acta Math. \textbf{123} (1969), 141--224. doi:10.1007/BF02392388.

\bibitem{Bartoszewicz}
A.~Bartoszewicz, \emph{On almost uniform convergence of some sequences of
states on infinite tensor products of von Neumann algebras},
Ann. Univ. Mariae Curie-Sk\l odowska Sect.~A \textbf{46} (1992), 1--4.
\href{https://www.bc.umcs.pl/Content/31206/PDF/czas4050_46_1992_1.pdf}{Available online}.


\bibitem{Black}
B.~Blackadar,
\emph{Operator Algebras: Theory of $C^*$-Algebras and
von Neumann Algebras},
Encyclopaedia of Mathematical Sciences 122,
Springer, Berlin, 2006.

\bibitem{BC} D. P. Blecher and I. Choi,
{\em A solution to Haagerup's problem: Normality of weights is commutatively determined,}
Preprint (2026), arXiv:2610.35774.

\bibitem{BLd} D. P. Blecher and L. E. Labuschagne, {\em Quantum integrals,}
Preprint (2026), arXiv:2610.06771v2 (15 October 2026).

\bibitem{BdIII} D. P. Blecher and L. E. Labuschagne, {\em Quantum integrals III,} (tentative title), Preprint in progress (2026)

\bibitem{BGL} D. P. Blecher, S. Goldstein and L. E. Labuschagne, {\em Abelian von Neumann algebras, measure algebras and $L^\infty$-spaces,}  Expositiones Mathematicae {\bf 40} (2022), 758--818 (61 pages),  DOI: 10.1016/j.exmath.2021.11.005

\bibitem{BHK} M. Bohata, J. Hamhalter and O. F. K. Kalenda, {\em  On Markushevich bases in preduals of von Neumann algebras,} Israel J. Math.\ {\bf 214} (2016), 867--884. DOI: 10.1007/s11856-016-1365-y.

\bibitem{BLS}
C. Budde, L. Labuschagne, and C. Steyn,
\emph{Almost everywhere convergence for noncommutative spaces},
Banach J. Math. Anal. \textbf{16}, article 56 (2022). 

\bibitem{CH} 
E.~Chetcuti and J.~Hamhalter,
\emph{The order topology on duals of $C^*$-algebras and von Neumann
algebras}, Studia Mathematica \textbf{254} (2020), 219--236, doi:10.4064/sm190108-11-7, arXiv:1812.04035v1.


\bibitem{CHW} E. Chetcuti,  J. Hamhalter, and H.  Weber, {\em The order topology for a von Neumann algebra,} Studia Math.\ 
{\bf  230} (2015), 95--120.

\bibitem{CMlocal}
V. I. Chilin and M. A. Muratov,
\emph{Comparison of topologies on $*$-algebras of locally measurable operators},
Positivity \textbf{17} (2013), 111--132.
DOI: 10.1007/s11117-011-0152-6.

\bibitem{Douglas}
R. G. Douglas,
\emph{On majorization, factorization, and range inclusion of operators on Hilbert space},
Proc. Amer. Math. Soc. \textbf{17} (1966), 413--415.

\bibitem{GGLN}
A. Garcia Guinto, M. Lorentz, and B. Nelson,
\emph{Murray--von Neumann dimension for strictly semifinite weights},
arXiv:2405.15725. 
\bibitem{GLbook}  S. Goldstein and L. E. Labuschagne, {\em Noncommutative measures and $L^p$ and Orlicz spaces, with applications to quantum physics,} Oxford Graduate Texts in Mathematics, 33. Oxford University Press, Oxford, (2025).

\bibitem{GL}
M.~Goldstein and S.~Litvinov,
\emph{Banach principle in the space of $\tau$-measurable operators},
Studia Math.\ \textbf{143} (2000), no.~1, 33--41.
doi:10.4064/sm-143-1-33-41.

\bibitem{Haag} U. Haagerup, {\em Normal weights on \(W^*\)-algebras}, J.\ Functional Analysis {\bf 19} (1975), 302--317. DOI: 10.1016/0022-1236(75)90060-9

\bibitem{Hamhalter} J. Hamhalter,
\emph{Quantum Measure Theory}, Fundamental Theories of Physics, 134,
Kluwer Academic Publishers, Dordrecht, 2003.
DOI: 10.1007/978-94-017-0119-8.

\bibitem{HaagerupOVW} U. Haagerup, {\em Operator valued weights  in von Neumann algebras, I},
J. Funct. Anal. \textbf{32} (1979), 175--206. 


\bibitem{HJX}
U. Haagerup, M. Junge, and Q. Xu,
\emph{A reduction method for noncommutative $L_p$-spaces and applications}, Trans. Amer. Math. Soc. \textbf{362}(4)(2010), 2125--2165.


\bibitem{HJKN} B. Hayes, D. Jekel, S. Kunnawalkam Elayavalli, B. Nelson, {\em General solidity phenomena and anticoarse spaces for type III$_1$ factors,} Preprint (2024), to appear J.\ Functional Analysis.

\bibitem{HJP} E. Hensz, R. Jajte and A. Paszkiewicz,
{\em The bundle convergence in von Neumann algebras and their \(L_2\)-spaces,}
Studia Mathematica {\bf 120} (1996), 23--46.

\bibitem{Hiai} F.\ Hiai, \emph{Concise lectures on selected topics of von Neumann algebras},
\href{https://arxiv.org/abs/2004.02383v1}{arXiv:2004.02383v1} (2020).

\bibitem{Hiaibk} F.\ Hiai, {\em Lectures on selected topics in von Neumann algebras,} EMS Series of Lectures in Mathematics. EMS Press, Berlin, 2021,

 \bibitem{Jaj} R. Jajte, {\em Strong limit theorems in noncommutative probability,} Lecture Notes in Math., 1110, Springer, Berlin-Heidelberg-New York-Tokyo, 1985.  
 
\bibitem{Lance} E. C. Lance,
\emph{Ergodic theorems for convex sets and operator algebras},
Invent. Math. \textbf{37} (1976), 201--214.
DOI: 10.1007/BF01390319.

 \bibitem{LGM}
B. Le Gac and F. M\'oricz,
\emph{Beppo Levi and Lebesgue type theorems for bundle
convergence in noncommutative $L_2$-spaces},
in \emph{Recent Advances in Operator Theory and Related Topics},
Operator Theory: Advances and Applications, vol.~127,
Birkh\"auser, Basel, 2001, pp.~447--464.
 
\bibitem{MKX}
M. S. Moslehian, M. Kian, and Q. Xu,
\emph{Positivity of $2\times2$ block matrices of operators}, Banach J.\ Math.\ Analysis {\bf 13} (2019), 726--743, arXiv:1904.08680.



\bibitem{Novikov} A. Novikov, \(L_1\)-\emph{space for a positive operator affiliated with von Neumann algebra,} Positivity {\bf 21} (2017), 359--375, DOI: 10.1007/s11117-016-0422-4

\bibitem{Pasz} A. Paszkiewicz, {\em Convergences in \(W^*\)-algebras,} J.\ Funct.\ Analysis {\bf 69} (1986), 143--154.

\bibitem{Petz}
D.~Petz, \emph{Almost uniform convergence on the predual of a von Neumann
algebra and an ergodic theorem}, in \emph{Aspects of Positivity in Functional
Analysis}, North-Holland Mathematics Studies, vol.~122, North-Holland,
Amsterdam, 1986, pp.~247--254. doi:10.1016/S0304-0208(08)71964-9.

\bibitem{Tak1} M. Takesaki, \emph{Theory of Operator Algebras I},
 Springer-Verlag, New York, 1979. 

\bibitem{TakesakiII}
M. Takesaki, \emph{Theory of Operator Algebras II},
Encyclopaedia of Mathematical Sciences, vol.~125,
Springer, Berlin, 2003.


\end{thebibliography}
\end{document}